\documentclass{amsart}

\usepackage{hyperref}
\usepackage{stmaryrd}
\usepackage{amssymb}

\newtheorem{defi}{Definition}[section]
\newtheorem{prop}[defi]{Proposition}
\newtheorem{thm}[defi]{Theorem}
\newtheorem{lem}[defi]{Lemma}
\newtheorem{cor}[defi]{Corollary}

\numberwithin{equation}{section}

\newcommand{\N}{\mathbb{N}}
\newcommand{\Z}{\mathbb{Z}}
\newcommand{\R}{\mathbb{R}}
\newcommand{\He}{\mathbb{H}}
\newcommand{\X}{\mathbb{X}}
\newcommand{\cA}{\mathcal{A}}

\newcommand{\cD}{\mathcal{D}}
\newcommand{\cH}{\mathcal{H}}
\newcommand{\cL}{\mathcal{L}}
\newcommand{\cP}{\mathcal{P}}
\newcommand{\cS}{\mathcal{S}}
\newcommand{\cY}{\mathcal{Y}}
\newcommand{\cc}{\text{cc}}
\newcommand{\Eucl}{\text{E}}
\newcommand{\Mass}{\operatorname{\mathbf M}}
\newcommand{\Norm}{\operatorname{\mathbf N}}

\newcommand{\Lip}{\operatorname{Lip}}

\newcommand{\Hol}{\operatorname{H}}

\newcommand{\spt}{\operatorname{spt}}
\newcommand{\diam}{\operatorname{diam}}

\newcommand{\defl}{\mathrel{\mathop:}=}

\newcommand{\im}{\operatorname{im}}

\newcommand{\B}{\operatorname B} 
\newcommand{\oB}{\operatorname U}  

\newcommand{\degr}[3]{\operatorname{deg}\left(#1,#2,#3\right)} 
\newcommand{\w}{\operatorname{w}} 

\newcommand{\overbar}[1]{\mkern 1.5mu\overline{\mkern-1.5mu#1\mkern-1.5mu}\mkern 1.5mu}

\begin{document}

\title{Some results on maps that factor through a tree}
\author{Roger Z\"{u}st}
\address{Institut de Math\'ematiques de Jussieu, B\^atiment Sophie Germain, 75205 Paris, France}
\email{roger.zuest@imj-prg.fr}
\thanks{Supported by the Swiss National Science Foundation.}

\begin{abstract}
We give a necessary and sufficient condition for a map defined on a simply-connected quasi-convex metric space to factor through a tree. In case the target is the Euclidean plane and the map is H\"older continuous with exponent bigger than 1/2, such maps can be characterized by the vanishing of some integrals over winding number functions. This in particular shows that if the target is the Heisenberg group equipped with the Carnot-Carath\'eodory metric and the H\"older exponent of the map is bigger than 2/3, the map factors through a tree.
\end{abstract}

\maketitle


\section{Introduction}

In these notes a \emph{tree} is a metric space $T$ that is uniquely arc-connected, i.e.\ for different points $p,p' \in T$ there is an embedding $\gamma : [0,1] \to T$ with $\gamma(0) = p$, $\gamma(1) = p'$ and any other such embedding is a reparameterization of $\gamma$. Let $\varphi : X \to Y$ be a uniformly continuous map between metric spaces. Depending on some conditions on $X$ we want to characterize those maps $\varphi$ that factor through a tree. We say that $\varphi$ has Property (T) if:
\begin{equation}
	\tag{T}
	\label{propertyt}
	\left\{ \
	\parbox{8.4cm}{For all $x,x' \in X$ with $\varphi(x) \neq \varphi(x')$ there is a point $y \in Y \setminus \{\varphi(x),\varphi(x')\}$ such that for any curve $\gamma: [0,1] \to X$ connecting $x$ with $x'$, $y$ is contained in $\im(\varphi \circ \gamma)$.}
	\right.
\end{equation}

Since a tree is uniquely arc-connected this property of $\varphi$ is necessary in order for it to factor through a tree. Depending on some conditions on $X$ it is also sufficient. To see that this doesn't work for any $X$ consider for example the unit circle in the complex plane and the map $x \mapsto x^2$. This map has Property~\eqref{propertyt} but it doesn't factor through a tree. If we additionally assume that the domain is simply-connected, this implication does hold. The terms used in the statements below will be clarified in the beginning of Section~\ref{section2}.

\begin{thm}
\label{treeintro}
Assume that $X$ is a $C$-quasi-convex metric space with $H_1(X) = 0$ or $H^{\Lip}_1(X) = 0$. Let $\varphi : X \to Y$ be a map that is $\sigma$-continuous and has Property~\eqref{propertyt}. Then there is a tree $(T,d_T)$ and surjective maps $\psi : X \to T$, $\overbar \varphi : T \to \im(\varphi)$ with $\varphi = \overbar \varphi \circ \psi$ and for all $x,x' \in X$,
\[
d_Y(\varphi(x),\varphi(x')) \leq d_T(\psi(x),\psi(x')) \leq \sigma(C d_X(x,x')) \, .
\]
The tree has the following additional properties:
\begin{enumerate}
	\item The metric $d_T$ is monotone on arcs, i.e.\ $d_T(p,p') \leq d_T(q,q')$ whenever $p$ and $p'$ are contained in the arc $[q,q']$ connecting $q$ with $q'$.
	\item $\dim(T,d_T) \leq 1$.
	\item For any $p \in T$ there is a contraction $\pi_p : T \times \R_{\geq 0} \to T$ with $\pi_p(q,t) \in [p,q]$, $\pi_p(q,0) = p$, $\pi_p(q,t) = q$ for $t \geq C\operatorname{dist}_X(\psi^{-1}(p),\psi^{-1}(q))$ and
	\[
	d_T(\pi_p(q,t),\pi_p(q',t')) \leq d_T(q,q') + \sigma(|t-t'|) \,.
	\]
\end{enumerate}
\end{thm}

A similar result has been obtained by Wenger and Young for Lipschitz maps in case $Y$ is purely $2$-unrectifiable \cite[Theorem~5]{WY}. Like in the proof presented therein, the tree in the theorem above is constructed as a quotient of $X$. In particular if $X$ is compact, then $T$ consists of the connected components of preimages of points, see Lemma~\ref{compactcase}.

Since $X$ is quasi-convex any curve in $X$ can be uniformly approximated by Lipschitz curves. In Property~\eqref{propertyt} we therefore could additionally assume that $\gamma$ is Lipschitz. The restriction to points $x,x' \in X$ with $\varphi(x) \neq \varphi(x')$ will be very convenient in Proposition~\ref{windvanishprop} where we give a condition on a H\"older continuous map which is equivalent to \eqref{propertyt} by using the theory of currents. If further $Y$ is the Euclidean plane and applying a connection between currents and winding numbers, we have the following characterization of Property~\eqref{propertyt}.

\begin{thm}
\label{hoelderthm}
Let $X$ be a quasi-convex metric space with $H_1^{\Lip}(X) = 0$ and $\varphi : X \to \R^2$ be a H\"older continuous map of regularity $\alpha$. If $\alpha > \frac{2}{3}$, then $\varphi$ has Property~\eqref{propertyt} if and only if for all closed Lipschitz curves $\gamma : S^1 \to X$,
\[
\int_{\R^2} \w_{\varphi \circ \gamma}(q) \, dq = 0 \,.
\]
If $\alpha > \frac{1}{2}$, then $\varphi$ has Property~\eqref{propertyt} if and only if for all closed Lipschitz curves $\gamma : S^1 \to X$,
\[
\int_{\R^2} \w_{\varphi \circ \gamma}(q) \, dq = \int_{\R^2} q_x \w_{\varphi \circ \gamma}(q) \, dq = \int_{\R^2} q_y\w_{\varphi \circ \gamma}(q) \, dq = 0 \,,
\]
where $q = (q_x,q_y)$. In both cases $\varphi$ factors through a tree as in Theorem~\ref{treeintro}.
\end{thm}

It was shown in \cite[Proposition~4.6]{Zt} that for a closed curve $\eta : S^1 \to \R^2$ of H\"older regularity $\alpha > \frac{1}{2}$, the winding number function $q \mapsto \w_{\eta}(q)$ is integrable and hence the integrals in the theorem above are well defined. Further, $\int \w_\eta^+$ and $\int \w_\eta^-$ are independent of the coordinate system in $\R^2$ in which we evaluate it because for an isometry $A$ of $\R^2$ it is $\w_{A\eta}(Aq) = \w_{\eta}(q)$. If we assume that $\int \w_\eta(q) \, dq = 0$, then the vector $\int q \w_\eta(q) \, dq$ has the geometric interpretation as $(c(\w_\eta^+) - c(\w_\eta^-))\int \w_\eta^\pm$, where $c(\w_\eta^+)$ and $c(\w_\eta^-)$ are the centers of mass of the densities $\w_\eta^+$ and $\w_\eta^-$ respectively. As such, the length of this vector and in particular the additional assumption for $\alpha > \frac{1}{2}$ in the statement above do also not depend on the coordinate system. These integrals of the winding number function are connected to the signature of curves as demonstrated in \cite[Theorem~1]{BNQ} for closed curves with bounded total variation. Actually, these terms represent the first few nontrivial entries in the logarithmic signature of the closed curves $\varphi \circ \gamma$. The theorem above follows from an integral formula for maps on a square that resembles similar formulas that appear in the theory of rough paths, see Theorem~\ref{roughsurface}.

Theorem~\ref{hoelderthm} is related to the H\"older problem for the Heisenberg group. Gromov showed in \cite{G} that there is no embedding of an open subset of the plane into the first Heisenberg group $\He$ equipped with the Carnot-Carath\'eodory metric $d_{\cc}$ that is H\"older continuous of regularity $\alpha > \frac{2}{3}$. The two theorems above strengthen this result.

\begin{thm}
	\label{heisenberggroupthm}
Let $X$ be a quasi-convex metric space with $H_1^{\Lip}(X) = 0$ and $\varphi : (X,d_X) \to (\He,d_{\cc})$ be a H\"older continuous map of regularity $\alpha > \frac{2}{3}$. Then $\varphi$ factors through a tree. If moreover $\dim(X) \geq 2$, then $\varphi$ can't be an embedding.
\end{thm}

A natural follow-up question is if the same conclusion also holds for $\alpha > \frac{1}{2}$. This would solve the H\"older problem for the Heisenberg group and show that there is no local homeomorphism from $\R^3$ to $\He$ of H\"older regularity $\alpha > \frac{1}{2}$. The statement is stronger than that and additionally would characterize these maps as locally factoring through a tree. A solution of this problem falls short because of the additional vanishing assumptions in Theorem~\ref{hoelderthm} for $\alpha > \frac{1}{2}$.

\section{Definitions and Preliminaries}
\label{section2}

Here we state some definitions and statements from the literature we will rely on and hope that the rest of these notes are reasonably self-contained.

A \emph{curve} in a topological space $X$ is a continuous map $\gamma : [a,b] \to X$ defined on some compact interval. An \emph{arc} in $X$ is the image of an injective curve. $X$ is called \emph{path-connected} if any two points in $X$ can be connected by a curve (being consistent with our terminology we should actually call it curve-connected, but this usage is not common). A metric space $X$ is called \emph{$C$-quasi-convex} if for any two points $x,x' \in X$ there is a curve $\gamma : [0,1] \to X$ connecting the two points and
\[
\operatorname{length}(\gamma) \leq C d_X(x,x') \,.
\]
By reparameterizing $\gamma$ by arc length we can assume that $\Lip(\gamma) \leq C d_X(x,x')$.

In metric geometry it is most common to work with geodesic trees. A \emph{geodesic tree} is a metric space in which any two points are connected by a unique geodesic, i.e.\ a curve of length equal to the distance of the two endpoints. Those trees we are working with here are more general but they are related to geodesic trees by a theorem of Mayer and Oversteegen.

\begin{thm}
[{\cite[Theorem~5.1]{MO}}]
\label{topthm}
If a metric space $T$ is uniquely arc-connected and locally arc-connected, then $T$ is homeomorphic to a geodesic tree.
\end{thm}

This in particular applies to the tree in Theorem~\ref{treeintro} because a tree with a monotone distance on arcs is locally arc-connected and by definition uniquely arc-connected. The dimension bound $\dim(T) \leq 1$ in Theorem~\ref{treeintro} holds for any of the three main topological dimensions; the small inductive dimension, the large inductive dimension and the Lebesgue covering dimension. This follows directly from the theorem above and \cite[Theorem~2.3]{MO}. More generally, a result of Lang and Schlichenmaier shows that geodesic trees with more than one point have Nagata dimension one \cite[Theorem~3.2]{LS} and this bounds from above the Lebesgue covering dimension \cite[Theorem~2.2]{LS} and hence also the other two topological dimensions mentioned earlier.

Let $\sigma : \R_{\geq 0} \to \R_{\geq 0}$ be a continuous and strictly increasing function with $\sigma(0) = 0$. A map $\varphi : X \to Y$ between metric spaces is \emph{$\sigma$-continuous} if for all $x,x' \in X$,
\[
d_Y(\varphi(x),\varphi(x')) \leq \sigma(d_X(x,x')) \,.
\]
Note that for any uniformly continuous map $\varphi : X \to Y$ there is such a $\sigma$ for which $\varphi$ is $\sigma$-continuous. For example,
\[
\omega(t) \defl \sup\{d_Y(\varphi(x),\varphi(x')) : d_X(x,x') \leq t\}
\]
is increasing and continuous at $0$ with $\omega(0) = 0$. The function $\sigma : \R_{\geq 0} \to \R_{\geq 0}$ defined by
\[
\sigma(t) \defl t + \frac{1}{t}\int_t^{2t} \omega(s) \, ds
\]
has all the properties we want and makes $\varphi$ into a $\sigma$-continuous map. With this observation we see that any uniformly continuous map is treated by Theorem~\ref{treeintro}. As a particular instance of $\sigma$-continuity, $\varphi : X \to Y$ is called \emph{H\"older continuous of regularity $\alpha > 0$} if there is a constant $C \geq 0$ such that for all $x,x' \in X$,
\[
d_Y(\varphi(x),\varphi(x')) \leq C d_X(x,x')^\alpha \,.
\]
The infimum over all such $C$ is $\Hol^\alpha(\varphi)$ and $\Hol^\alpha(X,Y)$ denotes the set of all H\"older continuous maps of regularity $\alpha$ from $X$ to $Y$. In case $Y = \R$ we abbreviate $\Hol^\alpha(X) \defl \Hol^\alpha(X,\R)$. For a sequence $(\varphi_k)$ in $\Hol^\alpha(X,Y)$ we write $\varphi_k \stackrel{\alpha}{\longrightarrow} \varphi$ if $\sup_{x \in X} d_Y(\varphi_k(x),\varphi(x)) \to 0$ and $\sup_k \Hol^\alpha(\varphi_k) < \infty$. It follows immediately that this limit satisfies $\Hol^\alpha(\varphi) \leq \liminf_k \Hol^\alpha(\varphi_n)$ and hence $\varphi \in \Hol^\alpha(X,Y)$. The following theorem collects some results of Young.

\begin{thm}
[{\cite{Y}}]
\label{StieltjesLemma}
Let $a \leq b$ and $0 < \alpha,\beta \leq 1$ with $\alpha + \beta > 1$. If $f \in \Hol^\alpha([a,b])$ and $g \in \Hol^\beta([a,b])$, then the Riemann-Stieltjes integral $\int_a^b f \, dg$ exists. Further:
\begin{enumerate}
	\item There is a constant $C_{\alpha,\beta}$, such that for all $c \in [a,b]$,
	\[
	\left|\int_a^b f \, dg - f(c)(g(b) - g(a))\right| \leq C_{\alpha,\beta}\Hol^\alpha(f)\Hol^\beta(g)|b-a|^{\alpha + \beta} \,.
	\]
	\item If $f$ and $g$ are Lipschitz, then
	\[
	\int_a^b f \, dg = \int_a^b f(t)\, g'(t) \, d\cL^1(t) \,.
	\]
	\item If $(f_k)$ and $(g_k)$ are sequences of functions on $[a,b]$ with $f_k \stackrel{\alpha}{\longrightarrow} f$ and $g_k \stackrel{\beta}{\longrightarrow} g$, then
	\[
	\int_a^b f_k \, dg_k \to \int_a^b f \, dg \,.
	\]
\end{enumerate}
\end{thm}

This Riemann-Stieltjes integral over H\"older functions can be generalized to higher dimensions. For a square $Q \subset \R^2$ we denote by $\cP_n(Q)$ the partition of $Q$ into $4^n$ similar squares. Given functions $f,g_1,g_2 : Q \to \R$ we define the approximate functionals
\[
I_{Q,n}(f,g_1,g_2) \defl \sum_{R \in \cP_n(Q)} f(p_R) \int_{\partial R} g_1 \, dg_2 \,,
\]
for some predefined choice of points $p_R \in R$ and assuming the integrals in the sum make sense. They are to be understood as Riemann-Stieltjes integrals running counterclockwise around the boundary of the indicated square. In particular, if $g_1$ and $g_2$ are H\"older continuous as in Theorem~\ref{StieltjesLemma}, then $I_{Q,n}(f,g_1,g_2)$ is well defined for all $n$. The limit we obtain below does not depend on the choice of the points $p_R \in R$ and we will thus not refer to them specifically (we can fix $p_R$ to be the barycenter of $R$ for example). The following lemma is the two-dimensional case of \cite[Theorem 3.2]{Z}.

\begin{lem}
\label{twodimcurrent}
Let $f \in \Hol^\alpha(Q)$, $g_1 \in \Hol^{\beta_1}(Q)$ and $g_2 \in \Hol^{\beta_2}(Q)$. If $\alpha + \beta_1 + \beta_2 > 2$, then the limit
\[
I_{Q}(f,g_1,g_2) \defl \lim_{n \to \infty} I_{Q,n}(f,g_1,g_2)
\]
exists. Further, $I_Q$ satisfies and is uniquely defined by the following properties:
\begin{enumerate}
	\item $I_{Q}$ is linear in each argument,
	\item $I_{Q}(f,g_1,g_2) = \int_Q f \det D(g_1,g_2) \, d\cL^2$ if all three functions are Lipschitz,
	\item $I_{Q}(f_k,g_{1,k},g_{2,k}) \to I_{Q}(f,g_{1},g_{2})$ if $f_k \stackrel{\alpha}{\longrightarrow} f$ and $g_{i,k} \stackrel{\beta_i}{\longrightarrow} f$ for $i=1,2$.
\end{enumerate}
\end{lem}

We will occasionally use properties of the mapping degree and the winding number respectively. Regarding those, everything we state here can be found for example in \cite{OR}. We mention here some relation between winding numbers and currents. An appropriate theory for currents in metric spaces was introduced by Ambrosio and Kirchheim in \cite{AK} extending the classical theory which is described with great detail in the monograph of Federer \cite{F}. Since for a large part we work with H\"older maps the currents that appear have in general infinite mass and we therefore mainly refer to the theory of Lang \cite{L} which does not rely on the finite mass assumption in its initial setting.

For $w \in L^1(\R^n)$ we write $\llbracket w \rrbracket$ for the current of finite mass in $\Mass_n(\R^n)$ obtained by integrating $n$-forms over $w$. For an oriented submanifold $M^m \subset \R^n$ we also denote by $\llbracket M \rrbracket \in \cD_m(\R^n)$ the $m$-dimensional current induced by integrating $m$-forms over $M$. If $X$ and $Y$ are metric spaces, $\varphi \in \Hol^\alpha(X,Y)$ for some $\alpha > \frac{n}{n+1}$ and $T \in \Norm_n(X)$ is a normal current with compact support, then $\varphi_\# T$ is a well defined current in $\cD_n(Y)$ by \cite[Theorem 4.3]{Z}. For $T = \llbracket S^1\rrbracket$ and $Y = \R^2$, there is some connection between the winding number function $q \mapsto \w_{\gamma}(q)$ of $\gamma : S^1 \to \R^2$ and the push-forward $\gamma_\# \llbracket S^1\rrbracket$ as noted in \cite[Proposition 4.6]{Zt}.

\begin{lem}
	\label{windingcurrents}
Let $\gamma \in \Hol^\alpha(S^1,\R^2)$ for $\alpha > \frac{1}{2}$. Then $\w_{\gamma}$ is integrable and $\llbracket \w_{\gamma} \rrbracket$ is the unique filling with compact support of $\gamma_\# \llbracket S^1 \rrbracket \in \cD_1(\R^2)$, i.e.\
\[
\partial \llbracket \w_{\gamma} \rrbracket = \gamma_\# \llbracket S^1 \rrbracket \,.
\]
Respectively, if $g = (g_1,g_2) \in \Lip(\R^2,\R^2)$, then
\[
\int_{\R^2} \w_{\gamma}(q) \det Dg(q) \, dq = \int_{S^1} g_1 \circ \gamma \, d(g_2 \circ \gamma) \,.
\]
If $\varphi \in \Hol^\alpha(Q,\R^2)$ for a square $Q \subset \R^2$ and $\alpha > \frac{2}{3}$ this can be combined with Lemma~\ref{twodimcurrent} to obtain,
\begin{equation}
\label{twothirdsformula}
\int_{\R^2} \w_{\varphi|\partial Q} f \det Dg  = \varphi_\# \llbracket Q \rrbracket(f\, dg_1 \wedge dg_2) = I_{Q}(f \circ \varphi, g_1 \circ \varphi, g_2 \circ \varphi) \,.
\end{equation}
for $f,g_1,g_2 \in \Lip(\R^2)$ and $g = (g_1,g_2)$.
\end{lem}

We denote by $H_m(X)$ the $m$th singular homology group and by $H_m^{\Lip}(X)$ the $m$th singular Lipschitz homology group. Every singular Lipschitz $m$-chain in $X$ represents an $m$-dimensional integral current, i.e.\ an element of $I_m(X)$, as defined in \cite{AK} or \cite{L}. More precisely if $c = \sum_i n_i \Gamma_i$ for a finite sum of integers $n_i$ and Lipschitz maps $\Gamma_i : \Delta^m \to X$ defined on the $m$-dimensional standard simplex $\Delta^m$, the current $\llbracket c \rrbracket \defl \sum_{i=1}^l n_i {\Gamma_i}_\# \llbracket \Delta^m \rrbracket$ is an element of $I_m(X)$. By definition it is clear that $\llbracket c + c'\rrbracket = \llbracket c\rrbracket + \llbracket c'\rrbracket$ and Stokes' theorem implies that $\llbracket \partial c \rrbracket = \partial \llbracket c \rrbracket$, see \cite{RS} for more details on this construction and some further results on the relation between homology groups and integral currents. In particular, if $H_1^{\Lip}(X)=0$ and $\gamma : S^1 \to X$ is Lipschitz, then there are finitely many Lipschitz maps $\Gamma_i : \B^2(0,1) \to X$ and integers $n_i$ such that 
\begin{equation}
\label{filling}
\gamma_\# \llbracket S^1 \rrbracket = \sum_i n_i \partial \left({\Gamma_i}_\# \llbracket \B^2(0,1) \rrbracket\right) \, .
\end{equation}

\section{Construction of the tree}

Let $\varphi : X \to Y$ be a continuous map between metric spaces $(X,d_X)$ and $(Y,d_Y)$ as in Theorem~\ref{treeintro}. In this section $d$ denotes the intrinsic metric on $X$. This means that $d(x,x')$ is the infimal length of all curves connecting $x$ with $x'$, see e.g.\ \cite[Chapter~2]{BBI} for properties of $d$. The resulting space $(X,d)$ is a length space, i.e.\ for any two points $x,x' \in X$ and any $\epsilon > 0$ there is a curve $\gamma_\epsilon$ with $\operatorname{length}(\gamma_\epsilon) \leq d(x,x') + \epsilon$ connecting $x$ and $x'$. Because $(X,d_X)$ is $C$-quasi-convex,
\begin{equation}
\label{geodesic}
d_X(x,x') \leq d(x,x') \leq C d_X(x,x') \,,
\end{equation}
for all $x,x' \in X$. Since $\sigma$ is increasing the first estimate shows that $\varphi$ is also $\sigma$-continuous with respect to $d$. Until the end of this section we work with the length metric $d$ instead of $d_X$.
\newline

Similarly as in the proof of \cite[Theorem~5]{WY} we define a pseudo-metric on $X$. For $x,x' \in X$,
\begin{equation}
\label{defD}
D(x,x') \defl \inf\left\{\diam(\varphi(C)) : x,x' \in C \text{ and } C \text{ is connected} \right\} \,.
\end{equation}

\begin{lem}
\label{lipestimate}
$D$ is a pseudo-metric on $X$ and moreover for all $x,x' \in X$,
\[
d_Y(\varphi(x),\varphi(x')) \leq D(x,x') \leq \sigma(d(x,x')) \,.
\]
\end{lem}

\begin{proof}
For connected subsets $A,B \subset X$ with $x,x' \in A$ and $x',x'' \in B$ there holds $\diam(\varphi(A \cup B)) \leq \diam(\varphi(A)) + \diam(\varphi(B))$ because $\varphi(A) \cap \varphi(B)$ is nonempty. Similarly, because $A\cap B$ is nonempty, $A \cup B$ is connected and this immediately implies the triangle inequality for $D$. The first inequality is obvious since any set $C' \subset Y$ that contains both $\varphi(x)$ and $\varphi(x')$ satisfies $\diam(C') \geq d_Y(\varphi(x),\varphi(x'))$. For the second, take $C_\epsilon \defl \im(\gamma_\epsilon)$ for some curve $\gamma_\epsilon$ with $\operatorname{length}(\gamma_\epsilon) \leq d(x,x') + \epsilon$ connecting $x$ and $x'$ in $X$. Since $\varphi$ is $\sigma$-continuous,
\[
D(x,x') \leq \diam(\varphi(C_\epsilon)) \leq \sigma(\diam(C_\epsilon)) \leq \sigma(d(x,x') + \epsilon) \,.
\]
Taking the limit for $\epsilon \to 0$, the continuity of $\sigma$ implies the second estimate.
\end{proof}

Our candidate for the tree in Theorem~\ref{treeintro} is the set of equivalence classes $T \defl X/_\sim$, where $x \sim x'$ if $D(x,x') = 0$. We further denote by $\psi : X \to T$ the quotient map $\psi(x) \defl [x]$ and define $\overbar \varphi : T \to Y$ by $\overbar \varphi([x]) \defl \varphi(x)$. Note that $\overbar \varphi$ is well defined by the lemma above. An obvious choice for a metric on $T$ is $D$ itself, i.e.\ $D([x],[x']) \defl D(x,x')$. In the next part we will show that $(T,D)$ is indeed a tree.

\subsection{Proof that \textit{T} is a tree}
It follows from Lemma~\ref{lipestimate} that every point $p \in T$ represents a closed subset of $\varphi^{-1}(y)$ for some $y \in Y$. In particular, any connected component of $\varphi^{-1}(y)$ is contained in some $p \in T$. In case $X$ is compact this is actually a characterization of $T$. Although we will not use this fact in the process, we think that it is interesting to note anyway and add a proof for completeness sake.

\begin{lem}
\label{compactcase}
If $X$ is compact, then
\[
T = \{c : c \text{ is a connected component of } \varphi^{-1}(y) \text{ for some } y \in Y \} \,.
\]
\end{lem}

\begin{proof}
As noted above, any connected component $c$ as in the statement is contained in some $p \in T$. On the other side let $c,c'$ be two such components with $D(x,x') = 0$ for some fixed $x \in c$ and $x' \in c'$. We want to show that $c = c'$. From Lemma~\ref{lipestimate} it follows that $\varphi(c) = \varphi(c') = \{y\}$ for some $y \in Y$ and from the definition of $D$ we obtain for any $n \in \N$ a connected subset $C_n \subset X$ with $x,x' \in C_n$ and $\varphi(C_n) \subset \B_Y(y,\frac{1}{n})$. By taking the closure, we can as well assume that $C_n$ is compact. By a theorem of Blaschke \cite{B}, the set
\[
\{K \subset X : K \text{ is compact and nonempty}\}
\]
equipped with the Hausdorff distance
\[
d_{\text{H}}(K,K') \defl \inf\{ \epsilon > 0 : K \subset \B_X(K',\epsilon), K' \subset \B_X(K,\epsilon) \}
\]
is a compact metric space. Applied to the situation at hand, there is a subsequence of $c \cup c' \cup C_n$ converging to some compact subset $C \subset X$. It is easy to check that a Hausdorff limit of connected sets is connected itself. $C$ is therefore compact and connected and moreover contains $c$ and $c'$. Since $\varphi$ is continuous we have $\varphi(C) = \{y\}$ and hence $c = c' = C$.
\end{proof}

As used in the proof above, it follows directly from the definition of $D$ that for any $x,x' \in X$ and any $\epsilon > 0$ there is a connected set $C_\epsilon \subset X$ with $x,x' \subset C_\epsilon$ and $\varphi(C_\epsilon) \subset \oB_Y(\varphi(x), D(x,x') + \epsilon)$. Because $(X,d)$ is a length space and $C_\epsilon$ is connected, any open neighborhood $\oB_X(C_\epsilon,\delta)$ of $C_\epsilon$ is Lipschitz path connected. Since $\varphi$ is uniformly continuous and by choosing $\delta$ small enough, there is a curve $\gamma_{x,x',\epsilon} : [0,1] \to X$ with $\gamma_{x,x',\epsilon}(0) = x$, $\gamma_{x,x',\epsilon}(1) = x'$ and
\begin{equation}
\label{connectset}
\im(\varphi \circ \gamma_{x,x',\epsilon}) \subset \oB_Y(\varphi(x),D(x,x') + \epsilon) \,.
\end{equation}

Given two points $x,x' \in X$ with $\varphi(x) \neq \varphi(x')$ let $\cY(x,x')$ be the set of all points $y \in Y$ such that for any curve $\gamma : [0,1] \to X$ connecting $x$ with $x'$, the point $y$ lies in $\im(\varphi \circ \gamma)$. Property~\eqref{propertyt} guarantees that apart from $\varphi(x)$ and $\varphi(x')$ the set $\cY(x,x')$ contains additional points. 

\begin{lem}
\label{connectcomp}
Let $\gamma_i : [0,1] \to X$, $i=1,2$, be two curves with $\gamma_1(t), \gamma_2(t) \in p_t \in T$ for $t = 0,1$ and $\overbar \varphi(p_0) \neq \overbar \varphi(p_1)$. Then $\cY(\gamma_1(0),\gamma_1(1)) = \cY(\gamma_2(0),\gamma_2(1))$.
\end{lem}

\begin{proof}
For some fixed $y \in \cY(\gamma_2(0),\gamma_2(1))$ we want to show that $y$ is also in $\cY(\gamma_1(0),\gamma_1(1))$. If $y = \overbar \varphi(p_0)$ or $y = \overbar \varphi(p_1)$ we are done. So assume this is not the case and let $\epsilon > 0$ be small enough such that
\[
\epsilon < \min\{d_Y(\overbar \varphi(p_0),y),d_Y(\overbar \varphi(p_1),y)\} \,.
\]
Using the curves as in \eqref{connectset}, define the concatenated curve (read from left to right)
\[
\gamma_1' \defl \gamma_{\gamma_2(0),\gamma_1(0),\epsilon} \ast \gamma_1 \ast \gamma_{\gamma_1(1),\gamma_2(1),\epsilon} \,.
\]
$\gamma_1'$ connects $\gamma_2(0)$ with $\gamma_2(1)$ by going through $\gamma_1$. For $t = 0,1$, \eqref{connectset} implies,
\[
\im(\varphi \circ \gamma_{\gamma_2(t),\gamma_1(t),\epsilon}) \subset \oB_Y(\overbar\varphi(p_t),\epsilon) \subset Y \setminus \{y\} \,.
\]
Since $y \in \im(\varphi \circ \gamma_1')$ by assumption, we get that $y \in \im(\varphi \circ \gamma_1)$.
\end{proof}

This lemma allows to define $\cY(p,p')$ for $p,p' \in T$ in case $\overbar \varphi(p) \neq \overbar \varphi(p')$. The next result is the main reason for this particular definition of $T$ and precisely where we need $H_1(X) = 0$ or $H^{\Lip}_1(X) = 0$. It can be stated in both the continuous and Lipschitz category and we therefore don't prefer one homology group over the other. Note that for an open set in a locally (Lipschitz) path connected space, components and (Lipschitz) path components are the same. It is likely that this result is classical, but the author couldn't find a reference for it.

\begin{lem}
\label{connectedcomp}
Let $Z$ be a connected and locally (Lipschitz) path connected space with $H^{(\Lip)}_1(Z) = 0$. Assume that $A \subset Z$ is a closed set that disconnects $z$ and $z'$ in $Z$. Then there is a connected component of $A$ that disconnects $z$ and $z'$.
\end{lem}

\begin{proof}
We will formulate the proof in the continuous category since the arguments in the Lipschitz case are the same. Consider the collection $\cA$ of closed subsets of $A$ that disconnect $z$ and $z'$. Set inclusion gives a partial order on $\cA$ and we want to show that there is a minimal element in $\cA$. By Zorn's lemma it suffices to find for any chain $\cA' \subset \cA$ a lower bound in $\cA$. Let $B \defl \bigcap \cA'$ and $C \subset X$ be the image of a curve that connects $z$ and $z'$. By definition, $B$ is closed and the intersection $C \cap A_1 \cap \cdots \cap A_n$ is nonempty for every finite collection $A_1,\dots,A_n \in \cA'$. This holds because $A_1 \cap \cdots \cap A_n = A_i$ for some $i$ since $\cA'$ is a chain and $A_i$ disconnects $z$ and $z'$ inside the connected set $C$. Since $C \cap A'$ is a nonempty closed set in the compact set $C$ for any $A' \in \cA'$, the intersection $C \cap B$ is nonempty too and hence $B \in \cA$.

So let $M$ be a minimal element of $\cA$. This $M$ has to be connected. Assume by contradiction that it is not and let $M_1$ and $M_2$ be a partition of $M$ into disjoint, nonempty, closed subsets. Set $U \defl X \setminus M_1$ and $V \defl X \setminus M_2$. Clearly, $X = U \cup V$ and the tail of the Mayer-Vietoris sequence reads as
\[
0 = H_1(X) \xrightarrow{\partial_*} H_0(U \cap V) \xrightarrow{(i_*,j_*)} H_0(U) \oplus H_0(V) \xrightarrow{k_* - l_*} H_0(X) = \Z,
\]
where $i : U\cap V \to U$, $j : U\cap V \to V$, $k : U \to X$ and $l : V \to X$ are the inclusions. Because $H_1(X) = 0$ and this sequence is exact, the homomorphism $(i_*,j_*)$ is injective. Since $z$ and $z'$ are disconnected by $M$, they represented different elements $[z]$ and $[z']$ in $H_0(U \cap V)$. It follows that $(i_*[z],j_*[z]) \neq (i_*[z'],j_*[z'])$ and hence $i_*[z] \neq i_*[z']$ or $j_*[z] \neq j_*[z']$. This means that $z$ and $z'$ are in different path components of $U$ or $V$, respectively, $M_1$ or $M_2$ disconnects $z$ and $z'$, contradicting the minimality of $M$. Therefore $M$ is connected and contained in some connected component of $A$.
\end{proof}

This result is used in the following lemma.

\begin{lem}
\label{disconnecter}
Let $p,p' \in T$ with $\overbar \varphi(p) \neq \overbar \varphi(p')$. Then for every $y \in \cY(p,p') \setminus \{\overbar \varphi(p), \overbar \varphi(p')\}$ there is some $q \in T$ with $\overbar \varphi(q) = y$ that disconnects $p$ from $p'$ inside $X$, i.e.\ every curve from $p$ to $p'$ in $X$ intersects $q$.
\end{lem}

\begin{proof}
Fix some points $x,x' \in X$ with $\psi(x) = p$ and $\psi(x') = p'$ and let $y \in \cY(x,x') \setminus \{\varphi(x), \varphi(x')\}$. By the definition of $\cY$, the two points $x$ and $x'$ are in different path components of $X \setminus \varphi^{-1}(y)$. As a length space $(X,d)$ is locally Lipschitz path connected and therefore Lemma~\ref{connectedcomp} implies that there is a connected component $c$ of $\varphi^{-1}(y)$ that disconnects $x$ and $x'$. By the construction of $T$, this set $c$ is contained in some $q \in T$ with $\overbar\varphi(q) = y$. Justified by Lemma~\ref{connectcomp} we identified $\cY(p,p') = \cY(x,x')$ and hence any curve connecting $p$ and $p'$ in $X$ intersects $q$.
\end{proof}

This result can be applied to curves in $T$ by constructing approximative lifts in $X$.

\begin{lem}
\label{lipapproximation}
Let $\gamma : [0,1] \to T$ be a curve connecting $p$ with $p'$ in $T$. Then there is a sequence of curves $\eta_n : [0,1] \to X$ such that $\psi\circ\eta_n(t) = \gamma(t)$ for $t = 0,1$ and $\psi\circ\eta_n$ converges uniformly to $\gamma$.

In particular, for all $p,p' \in T$ with $\overbar \varphi(p) \neq \overbar \varphi(p')$ there is some $q \in T \setminus \{p,p'\}$ such that any curve in $T$ connecting $p$ with $p'$ goes through $q$.
\end{lem}

\begin{proof}
Since $\gamma$ is uniformly continuous, we can find for every $\delta > 0$ some $m \in \N$ such that $D(\gamma(t),\gamma(t')) < \delta$ if $|t-t'| \leq \frac{1}{m}$. In particular, for every $0 \leq i \leq m-1$ it is $D(\gamma(\frac{i}{m}),\gamma(\frac{i+1}{m})) < \delta$. Fix some points $x_i \in \gamma(\frac{i}{m}) \subset X$ and using the curves of \eqref{connectset} define the curve $\eta_\delta : [0,1] \to X$ by
\[
\eta_\delta(t) \defl \gamma_{x_{i},x_{i+1},\delta}(tm - i) \quad \mbox{for } 0 \leq i \leq m-1 \mbox{ and } t \in \left[\tfrac{i}{m},\tfrac{i+1}{m}\right] \,.
\]
It is stated in \eqref{connectset} that 
\[
\im\bigl(\varphi \circ \gamma_{x_{i},x_{i+1},\delta}\bigr) \subset \oB_Y\bigl(\overbar\varphi(\gamma(\tfrac{i}{m})),D(x_i,x_{i+1}) + \delta\bigr) \subset \oB_Y\bigl(\overbar\varphi(\gamma(\tfrac{i}{m})),2\delta\bigr) \,.
\]
Since images of curves are connected, the definition of $D$ implies for $t \in \left[\tfrac{i}{m},\tfrac{i+1}{m}\right]$,
\begin{align*}
\psi \circ \eta_\delta(t) \in \im\bigl(\psi \circ \gamma_{x_{i},x_{i+1},\delta}\bigr) \subset \oB_T\bigl(\gamma(\tfrac{i}{m}),4\delta\bigr) \,.
\end{align*}
Hence for $0 \leq i \leq m-1$ and $t \in \left[\frac{i}{m},\frac{i+1}{m}\right]$,
\[
D(\psi \circ \eta_\delta(t),\gamma(t)) \leq D(\psi \circ \eta_\delta(t),\gamma(\tfrac{i}{m})) + D(\gamma(\tfrac{i}{m}),\gamma(t)) \leq 5\delta \,.
\]
Choosing $\delta = \frac{1}{n}$ gives a sequence of curves as stated in the first part of the lemma.

For the second part let $p,p' \in T$ with $\overbar \varphi(p) \neq \overbar \varphi(p')$ and $\gamma : [0,1] \to T$ be a curve connecting $p$ with $p'$ in $T$. From Lemma~\ref{disconnecter} and Property~\eqref{propertyt} it follows that there is some $q \in T \setminus \{p,p'\}$ disconnecting $p$ and $p'$ in $X$. Now let $\gamma : [0,1] \to T$ be any curve connecting $p$ with $p'$. From the first part we obtain a sequence of curves $\eta_n : [0,1] \to X$ such that $\psi\circ\eta_n(0) = p$, $\psi\circ\eta_n(1) = p'$ and $\psi\circ\eta_n$ converges uniformly to $\gamma$. Since $q$ disconnects $p$ and $p'$ in $X$, every curve of the sequence $\eta_n$ intersects $q$. Hence $q \in \im(\psi\circ\eta_n)$ for all $n$ and thus $q \in \im(\gamma)$ because $\psi\circ\eta_n$ converges uniformly to $\gamma$. This holds for any such curve $\gamma$ and the lemma follows.
\end{proof}

So far we have only considered points in $T$ with different images when applied to $\overbar \varphi$. The next lemma justifies this assumption and shows that curves in $T$ are not completely degenerate in some sense.

\begin{lem}
\label{nondegenerate}
If $\gamma : [0,1] \to T$ is a non constant curve, then there is a $t \in [0,1]$ with $\overbar \varphi(\gamma(0)) \neq \overbar \varphi(\gamma(t))$.
\end{lem}

\begin{proof}
Assume that $y = \overbar \varphi(\gamma(0)) = \overbar \varphi(\gamma(t))$ for all $t$. From Lemma~\ref{lipapproximation} it follows that there is a sequence of curves $\eta_n : [0,1] \to X$ with $\psi(\eta_n(t)) = \gamma(t)$ for $t = 0,1$ and $\psi \circ \eta_n$ converges uniformly to $\gamma$. Since $\overbar \varphi : T \to Y$ is uniformly continuous we get for any $\epsilon > 0$ that $\im(\varphi \circ \eta_n) \subset \oB_Y(y,\epsilon)$ for large enough $n$. For such an $n$ the definition of $D$ implies $D(\eta_n(0),\eta_n(t)) \leq 2\epsilon$ for all $t \in [0,1]$. Hence
\begin{align*}
D(\gamma(0),\gamma(t)) & \leq \limsup_{n\to\infty} D(\psi\circ\eta_n(0),\psi\circ\eta_n(t)) + D(\psi\circ\eta_n(t),\gamma(t)) \\
 & = \limsup_{n\to\infty} D(\eta_n(0),\eta_n(t)) + D(\psi\circ\eta_n(t),\gamma(t)) \leq 2\epsilon \,.
\end{align*}
This is true for any $\epsilon$ and it follows that $\gamma$ is constant.
\end{proof}

Now we are ready to show that $(T,D)$ is a tree.

\begin{prop}
	\label{treeprop}
$(T,D)$ is a tree. I.e.\ $(T,D)$ is an arc-connected metric space and if $\gamma_1,\gamma_2 : [0,1] \to T$ are two injective curves with the same endpoints, then the curves are reparameterizations of each other.
\end{prop}

\begin{proof}
As an image of a path-connected space, $T$ is also path-connected. It is a standard fact that such a topological space is arc-connected (this is also a consequence of Lemma~\ref{makeinjective}). Let $\gamma_1$ and $\gamma_2$ be two injective curves as in the statement. We will show that $\im(\gamma_1) = \im(\gamma_2)$. Assume by contradiction that there is a $t \in [0,1]$ with $\gamma_1(t) \notin \im(\gamma_2)$. By continuity there are $t_1 < t < t_2$ such that $\gamma_1(t) \notin \im(\gamma_2)$ for all $t \in \left]t_1,t_2\right[$ but $\gamma_1(t_i) \in \im(\gamma_2)$ for $i=1,2$. Concatenating the part of $\gamma_1$ from $t_1$ to $t_2$ with the backward parameterization of $\gamma_2$ connecting $\gamma_1(t_1)$ and $\gamma_1(t_2)$ we get an injective closed curve $\gamma : S^1 \to T$. By Lemma~\ref{nondegenerate} there are two poins $s,s' \in S^1$ with $\overbar\varphi(\gamma(s)) \neq \overbar\varphi(\gamma(s'))$. The second part of Lemma~\ref{lipapproximation} implies that there is some $q \in T\setminus \{\gamma(s),\gamma(s')\}$ disconnecting $\gamma(s)$ and $\gamma(s')$ inside $T$. Hence $\gamma$ has to go twice through $q$, a contradiction. This shows that $\im(\gamma_1) \subset \im(\gamma_2)$ and vice versa. Both curves are injective, therefore $\gamma_1 \circ \gamma_2^{-1}$ is a homeomorphism of $[0,1]$ and a reparameterization from $\gamma_2$ to $\gamma_1$.
\end{proof}

\subsection{Monotone metric and sigma-variation}
\label{monotone}

For $p,p' \in T$ we denote by $[p,p']$ a parameterization of the arc in $T$ connecting $p$ with $p'$. By abuse of notation we also use $[p,p']$ for the image of this curve, the arc itself. The metric $D$ may not be monotone on arcs. For this reason we introduce a new metric $d_T$ on $T$ defined by
\[
d_T(p,p') \defl \sup \{ D(q,q') : [q,q'] \subset [p,p']\} \,.
\]
It is not so hard to check that $d_T$ is indeed a metric on $T$. This follows from the fact that arcs are compact and for all $p,p',p'' \in T$ the arc $[p,p'']$ is contained in the union $[p,p'] \cup [p',p'']$.

\begin{lem}
\label{lipestimate2}
For all $x,x' \in X$,
\[
d_Y(\varphi(x),\varphi(x')) \leq d_T(\psi(x),\psi(x')) \leq \sigma(d(x,x')) \,.
\]
Further, $(T,d_T)$ is a tree.
\end{lem}

\begin{proof}
It is $D \leq d_T$ by the definition of $d_T$. Hence $\operatorname{id}_T: (T,d_T) \to (T,D)$ is continuous and the first inequality follows from Lemma~\ref{lipestimate}. To obtain the second, let $\epsilon > 0$ and $\gamma_\epsilon : [0,1] \to X$ be a curve in $X$ connecting $x$ and $x'$ with $\operatorname{length}(\gamma_\epsilon) \leq d(x,x') + \epsilon$. Let $p,p' \in [\psi(x),\psi(x')] \subset T$ with $D(p,p') = d_T(\psi(x),\psi(x'))$. Because $(T,D)$ is a tree, the curve $\psi \circ \gamma_\epsilon$ goes through both $p$ and $p'$, respectively, $\gamma_\epsilon$ intersects both $p$ and $p'$ seen as subsets of $X$. As noted in the beginning of this section, $\varphi$ is $\sigma$-continuous with respect to $d$ and hence the definition of $D$ in \eqref{defD} and Lemma~\ref{lipestimate} imply
\begin{align*}
d_T(\psi(x),\psi(x')) & = D(p,p') \leq \diam_Y(\varphi(\im(\gamma_\epsilon))) \leq \sigma(\diam_X(\im(\gamma_\epsilon))) \\
 & \leq \sigma(\operatorname{length}(\gamma_\epsilon)) \leq \sigma(d(x,x') + \epsilon) \,.
\end{align*}
Because $\epsilon > 0$ is arbitrary and $\sigma$ is continuous, $d_T(\psi(x),\psi(x')) \leq \sigma(d(x,x'))$.

Since $X$ is path-connected and $\psi : (X,d) \to (T,d_T)$ is continuous as we have just seen, any two points in $(T,d_T)$ can be connected by an arc. Because of $D \leq d_T$ any arc in $(T,d_T)$ is also an arc in $(T,D)$, hence up to reparameterization there can only be one arc in $(T,d_T)$ connecting two given points. Hence $(T,d_T)$ is a tree.
\end{proof}

Let $(Z,d_Z)$ be a metric space. The \emph{$\sigma$-variation} of a curve $\gamma : [a,b] \to X$ is defined by
\[
V_\sigma(\gamma) \defl \sup \sum_{i = 0}^{n-1} \sigma^{-1}(d_Z(\gamma(t_{i+1}),\gamma(t_i))) \,,
\]
where the supremum is taken over all finite sequences $a = t_0 < \cdots < t_n = b$. This definition is clearly independent of the parameterization of $\gamma$. One can show that for a continuous curve $\gamma : [a,b] \to (Z,d_Z)$ with $V_\sigma(\gamma) < \infty$ there is a reparameterization $\tilde \gamma : [0,V_\sigma(\gamma)] \to Z$ of $\gamma$ with $t = V_\sigma(\tilde \gamma|_{[0,t]})$. This is a standard result and uses the fact that $\tau(t) \defl V_\sigma(\gamma|_{[0,t]})$ is continuous. For the readers convenience we include a proof here.

\begin{lem}
Let $\nu : [0,1]^2 \to [0,\infty[$ and define $\overbar \nu : [0,1] \to [0,\infty]$ by
\[
\overbar \nu(t) \defl \sup \sum_{i=0}^{n-1} \nu(t_{i+1},t_i) \,,
\]
where the supremum is taken over all finite sequences $0 = t_0 < \cdots < t_n = t$. If $\nu$ is continuous, $\nu(t,t) = 0$ for all $t$ and $\overbar \nu(1) < \infty$, then $\overbar \nu$ is continuous.
\end{lem}

\begin{proof}
We will first show continuity at $0$. Obviously, $\overbar \nu(0) = 0$ and $\overbar\nu$ is increasing by definition. So it is enough to find a strictly decreasing sequence $(t_n)$ with $\lim_{n \to \infty} t_n = 0 = \lim_{n \to \infty} \nu(t_n)$. Without loss of generality we assume that $\overbar \nu(t) > 0$ for all $t > 0$. We start with $t_1 = 1$ and proceed recursively as follows. Given $t_n$, let $0 = t_n^0 < \dots < t_n^{k_n} = t_n$ be a strictly increasing sequence with
\[
\sum_{i=0}^{k_n-1} \nu(t_n^{i+1},t_n^i) > \frac{\overbar \nu(t_n)}{2} \,.
\]
Because $\lim_{a \to 0} \nu(b,a) + \nu(a,0) = \nu(b,0)$ for all $b$ we can assume that in the sequence above we have $0 < t_n^1 < \frac{t_n}{2}$ and 
\[
\sum_{i=1}^{k_n-1} \nu(t_n^{i+1},t_n^i) > \frac{\overbar \nu(t_n)}{2} \,.
\]
Note that the sum here starts from $i=1$. Set $t_{n+1} \defl t_n^1$. Obviously, $(t_n)$ is strictly decreasing and converges to $0$. Further, for all $l \in \N$
\[
\sum_{n=1}^l \overbar \nu(t_n) < 2 \sum_{n=1}^l \sum_{i=1}^{k_n-1} \nu(t_n^{i+1},t_n^i) \leq 2 \overbar \nu(1) \,.
\]
Because $\overbar \nu(1) < \infty$ we get that $\overbar \nu(t_n) \to 0$ for $n \to \infty$ and hence $\overbar \nu$ is continuous at $0$.

Now let $t > 0$ and we will show that $\overbar \nu$ is continuous from below at $t$. For any $n \in \N$ with $t > \frac{1}{n}$ let $0 = t_0 < \dots < t_{k_n} = t$ be a finite sequence with
\[
\sum_{i=0}^{k_n-1} \nu(t_{i+1},t_i) > \overbar \nu(t) - \frac{1}{n} \,.
\]
Because $\lim_{a \to t} \nu(b,a) + \nu(a,t) = \nu(b,t)$ for all $b$ we can assume that $\nu(t,t_{k_n-1}) < \frac{1}{n}$ and $0 < t - \frac{1}{n} < t_{k_n-1} < t$. Hence,
\[
\overbar \nu(t) \geq \overbar \nu(t_{k_n-1}) \geq \sum_{i=0}^{k_n-2} \nu(t_{i+1},t_i) > \overbar \nu(t) - \frac{2}{n} \,.
\]
We obtain $\lim_{n \to \infty} t_{k_n-1} = t$ and $\lim_{n \to \infty} \overbar \nu(t_{k_n-1}) = \overbar \nu(t)$. Since $\overbar \nu$ is increasing, this shows that it is continuous from below at $t$.

To see continuity from above, let $t < 1$ and $(t_n)$ be a strictly decreasing sequence with $\lim_{n \to \infty} t_n = t$. Without loss of generality we may assume that $\overbar \nu(t_n) > 0$ for all $n$. Let $0 = t_n^0 < \dots < t_n^{k_n} = t_n$ be a finite sequence with
\begin{equation}
\label{bound}
\sum_{i=0}^{k_n-1} \nu(t_n^{i+1},t_n^i) > \overbar \nu(t_n) - \frac{1}{n} \,.
\end{equation}
Since $t_n > t$ for each $n$, there is an index $i_n < k_n$ with $t \in [t_n^{i_n},t_n^{i_n + 1}[$. From $\lim_{n \to \infty} t_n = t$ it follows that $t_n^{i_n+1} > t$, $\lim_{n \to \infty} t_n^{i_n+1} \to t$ and since $\nu$ is uniformly continuous (the domain of definition is compact), 
\[
\lim_{n \to \infty} \nu(t_n^{i_n+1},t) + \nu(t,t_n^{i_n}) - \nu(t_n^{i_n+1},t_n^{i_n}) = 0 \,.
\]
For big $n$ we can therefore assume that \eqref{bound} is satisfied with $t$ being part of the sequence, say $t = t_n^{l_n}$ for some integer $l_n$. But then
\[
\overbar \nu(t_n) < \frac{1}{n} + \overbar \nu(t) + \sum_{i=l_n}^{k_n-1} \nu(t_n^{i+1},t_n^i) \,.
\]
This latter sum is over a partition of $[t,t_n]$ and as such tends to zero for $t_n \to t$ as we have already seen in the first part of the proof. Thus $\lim_{n \to \infty} \overbar \nu(t_n) = \overbar \nu(t)$ and $\overbar \nu$ is continuous from above at $t$.
\end{proof}

If $V_\sigma(\gamma) < \infty$ for $\gamma : [a,b] \to Z$, then $\tau : [a,b] \to [0,V_\sigma(\gamma)]$ defined by $\tau(v) \defl V_\sigma(\gamma|_{[0,v]})$ is increasing and continuous by the lemma above applied to the function $\nu(u,v) \defl \sigma^{-1}(d_Z(\gamma(u),\gamma(v)))$ defined on $[a,b]^2$. For $0 \leq s \leq t \leq V_\sigma(\gamma)$ let $a \leq u \leq v \leq b$ be such that $\tau(u) = s$ and $\tau(v) = t$. This is possible precisely because $\tau$ is continuous using the intermediate value theorem. Then
\begin{align*}
\sigma^{-1}(d_Z(\gamma(v),\gamma(u))) & \leq V_\sigma(\gamma|_{[u,v]}) \leq V_\sigma(\gamma|_{[0,v]}) - V_\sigma(\gamma|_{[0,u]}) \\
 & = \tau(v) - \tau(u) = t - s \,.
\end{align*}
This allows to define $\tilde \gamma : [0,V_\sigma(\gamma)] \to Z$ by $\tilde\gamma(\tau(v)) \defl \gamma(v)$ for all $v \in [a,b]$ and we obtain
\begin{equation}
\label{reparam}
d_Z(\tilde\gamma(t),\tilde\gamma(s)) \leq \sigma(t-s) \,.
\end{equation}
The new curve $\tilde \gamma$ is a continuous reparameterization of $\gamma$ and it follows from the definition of $\sigma$-variation that
\begin{equation}
\label{reparam2}
V_\sigma(\tilde \gamma|_{[0,t]}) = V_\sigma(\gamma|_{[0,v]}) = \tau(v) = t \,.
\end{equation}

For a curve $\gamma : [0,1] \to Z$ into a metric space $Z$ let $\cS_\gamma$ be the collection of all compact intervals $[a,b] \subset [0,1]$ with $a < b$, $\gamma(a) = \gamma(b)$ and for which there is no interval $[a',b']$ strictly containing $[a,b]$ with $\gamma(a') = \gamma(b')$. For a collection of disjoint intervals $\cS \subset \cS_\gamma$ define $\gamma_\cS : [0,1] \to Z$ by
\[
\gamma_\cS(t) \defl
\left\{
\begin{array}{ll}
\gamma(a)  & \mbox{if } t \in [a,b] \in \cS \,, \\
\gamma(t) & \mbox{otherwise} \,.
\end{array}
\right.
\]

\begin{lem}
	\label{makeinjective}
Let $\gamma : [0,1] \to Z$ be a $\sigma$-continuous curve and $\cS \subset \cS_\gamma$ be some maximal subset of disjoint intervals. Then $\gamma_\cS$ is $\sigma$-continuous and $\gamma_\cS(s) = \gamma_\cS(t)$ for $0 \leq s < t \leq 1$ implies that $\gamma_\cS$ is constant on $[s,t]$.

Moreover, $V_\sigma(\gamma_\cS) \leq 1$ and the reparameterization $\tilde \gamma_\cS : [0,V_\sigma(\gamma_\cS)] \to Z$ with respect to $\sigma$-variation is injective.
\end{lem}

\begin{proof}
Let $S \defl \bigcup \cS \subset [0,1]$. For a point $u \in S$ denote by $[a_u,b_u] \in \cS$ the unique interval with $u \in [a_u,b_u]$. Now let $s,t \in [0,1]$ with $s < t$. If both $s$ and $t$ are not in $S$, then
\[
d_Z(\gamma_\cS(t),\gamma_\cS(s)) = d_Z(\gamma(t),\gamma(s)) \leq \sigma(t - s) \,.
\]
If $s,t \in S$ we have three cases. If the two intervals $[a_s,b_s]$ and $[a_t,b_t]$ intersect they are the same and $\gamma_\cS(t) = \gamma_\cS(s)$. Otherwise, $b_s < a_t$ and since $\sigma$ is increasing
\[
d_Z(\gamma_\cS(t),\gamma_\cS(s)) = d_Z(\gamma(b_s),\gamma(a_t)) \leq \sigma(a_t - b_s) \leq \sigma(t - s) \,.
\]
If $s \in S$ and $t \notin S$, then
\[
d_Z(\gamma_\cS(t),\gamma_\cS(s)) = d_Z(\gamma(t),\gamma(b_s)) \leq \sigma(t - b_s) \leq \sigma(t - s) \,.
\]
The case $s \notin S$ and $t \in S$ is treated analogously.

Now assume that $\gamma_\cS(s) = \gamma_\cS(t)$ for $s < t$. If both $s$ and $t$ are in the complement of $S$, then $\gamma(s) = \gamma(t)$ and there is some $[a,b] \in \cS$ that intersects $[s,t]$ by the maximality of $\cS$. Neither $s$ or $t$ can be contained in $[a,b]$ because $s,t \notin S$, so $[a,b]$ is a proper subset of $[s,t]$, but this is not possible by the definition of $\cS_\gamma$. If $s \in S$ and $t \notin S$, then $b_s < t$ and
\[
\gamma(a_s) = \gamma_\cS(s) = \gamma_\cS(t) = \gamma(t) \,,
\]
contradicting $[a_s,b_s] \in \cS_\gamma$. The case $s \notin S$ and $t \in S$ is treated analogously. If $s,t \in S$ then $\gamma(a_s) = \gamma(b_t)$ and hence $[a_s,b_s] = [a_t,b_t]$ by the maximality of these intervals in $\cS_\gamma$.

For the last part first note that $V_\sigma(\gamma_\cS) \leq 1$ follows directly from the definition of $\sigma$-variation and the fact that $\gamma_\cS$ is $\sigma$-continuous. Let $0 \leq s \leq t \leq 1$ with $\tilde \gamma_\cS(s) = \tilde \gamma_\cS(t)$ and let $s = \tau(u)$ and $t = \tau(v)$ where $\tau(w) = V_\sigma(\gamma_\cS|_{[0,w]})$ and $\tilde \gamma_\cS(\tau(w)) = \gamma_\cS(w)$ for $w \in [0,1]$ as before. Then $\gamma_\cS(u) = \gamma_\cS(v)$ and from the first part we conclude that $\gamma_\cS$ is constant on $[u,v]$. This implies 
\[
s = \tau(u) = V_\sigma(\gamma_\cS|_{[0,u]}) = V_\sigma(\gamma_\cS|_{[0,v]}) = \tau(v) = t \,,
\]
and hence $\tilde \gamma_\cS$ is injective.
\end{proof}

Let $p,p' \in T$ be two different points and for a fixed $\epsilon > 0$ consider a curve $\eta : [0,1] \to X$ with $\operatorname{length}(\eta) \leq \operatorname{dist}_{(X,d)}(p,p') + \epsilon$ connecting the subsets $p$ and $p'$ of $X$. Reparameterizing $\eta$ linearly with respect to arc length we may assume that $\Lip(\eta) \leq \operatorname{dist}_{(X,d)}(p,p') + \epsilon$. From Lemma~\ref{lipestimate2} it follows that $\gamma \defl \psi \circ \eta$ satisfies
\[
d_T(\gamma(t),\gamma(s)) \leq \sigma(d(\eta(t),\eta(s))) \leq \sigma((\operatorname{dist}_{(X,d)}(p,p') + \epsilon) |t-s|) \,.
\]
The curve $\gamma_{\cS} : [0,1] \to (T,d_T)$ constructed in the lemma above for some maximal set $\cS \subset \cS_\gamma$ has the same continuity estimate as $\gamma$ and satisfies $\im(\gamma_{\cS})=[p,p']$ because of the second part of Lemma~\ref{makeinjective} and the fact that $T$ is a tree. Hence
\begin{equation}
\label{projectiongeoesic}
V_\sigma([p,p']) = V_\sigma(\gamma_{\cS}) \leq \operatorname{dist}_{(X,d)}(p,p') + \epsilon \,.
\end{equation}
With \eqref{reparam}, \eqref{reparam2}, Lemma~\ref{makeinjective} and taking the limit $\epsilon \to 0$ in \eqref{projectiongeoesic}, the reparameterization of $\gamma_\cS$ with respect to the $\sigma$-variation gives a curve $[p,p'] : [0,V_\sigma([p,p'])] \to (T,d_T)$ with the following properties:
\begin{equation}
\label{curveintree}
\left\{
\begin{array}{l}
d_T([p,p'](t),[p,p'](s)) \leq \sigma(|t-s|) \,, \\
V_\sigma([p,p']|_{[0,t]}) = t \,, \\
V_\sigma([p,p']) \leq \operatorname{dist}_{(X,d)}(p,p') \,, \\
\mbox{}[p,p'] \mbox{ is injective} \,.
\end{array}
\right.
\end{equation}
Note that a priori the curves $[p,p']$ we construct depend on $\gamma$ and $\cS$, but the second and forth property above uniquely determine a parameterization of the arc from $p$ to $p'$.

\begin{proof}[Proof of Theorem~\ref{treeintro}]
From Lemma~\ref{lipestimate2} and the definition of $d_T$ we get that $(T,d_T)$ is a tree with a metric monotone on arcs. From Theorem~\ref{topthm} and the discussion thereafter we see that $\dim(T,d_T) \leq 1$ (and of course $\dim(T,d_T) = 1$ if $T$ contains more than one point). The maps $\psi$ and $\overbar \varphi$ have the right continuity properties with respect to the length metric $d$ on $X$. By translating the estimates in Lemma~\ref{lipestimate2} back to the original metric $d_X$ using \eqref{geodesic} we obtain the continuity estimates in the statement of the theorem. It remains to construct the contractions.

Fix a point $p \in T$ and consider the map $\pi_p : T \times \R_{\geq 0} \to T$ defined by
\[
\pi_p(q,t) \defl [p,q](\min\{V(q),t\}) \,,
\]
where $V(q) \defl V_\sigma([p,q])$. If $t \geq C \operatorname{dist}_{(X,d_X)}(p,q)$, then $t \geq \operatorname{dist}_{(X,d)}(p,q)$ because $(X,d_X)$ is $C$-quasi-convex and $\pi_p(q,t) = q$ follows from \eqref{curveintree}. For $t_1,t_2 \geq 0$ and $q \in T$ it also follows from \eqref{curveintree} that
\[
d_T(\pi_p(q,t_1),\pi_p(q,t_2)) \leq \sigma(|\min\{V(q),t_1\} - \min\{V(q),t_2\}|) \leq \sigma(|t_1-t_2|) \,.
\]
Since $T$ is a tree, there is a unique point $q \in T$ in the intersection of the images of the arcs $[q_1,q_2]$, $[p,q_1]$ and $[p,q_2]$ for all choices of $q_1,q_2 \in T$. At equal times we have $\pi_p(q_1,t) = \pi_p(q_2,t)$ if $t \leq V(q)$ and $[\pi_p(q_1,t),\pi_p(q_2,t)]$ is contained in $[q_1,q_2]$ otherwise. Because $d_T$ is monotone on arcs, this leads to
\[
d_T(\pi_p(q_1,t),\pi_p(q_2,t)) \leq d_T(q_1,q_2) \,,
\]
for all $t \geq 0$. Combining the two estimates using the triangle inequality for $d_T$ we get
\[
d_T(\pi_p(q_1,t_1),\pi_p(q_2,t_2)) \leq d_T(q_1,q_2) + \sigma(|t_1-t_2|) \,.
\]
This finishes the proof of Theorem~\ref{treeintro}.
\end{proof}

We want to mention some implications of these contractions for obtaining continuous extensions.

\begin{cor}
	\label{extensioncor}
Let $f : S^{m-1} \to T$ for $m \geq 2$ and $L \geq 0$ be a constant such that $d_T(f(s),f(s')) \leq \sigma(L|s-s'|)$ for all $s,s' \in S^{m-1}$. Then there is an extension $F : \B^{m}(0,1) \to T$ such that $\im(f) = \im(F)$ and
\[
d_T(F(x),F(x')) \leq 2\sigma(2\pi L|x-x'|) \,,
\]
for all $x,x' \in \B^m(0,1)$.
\end{cor}

\begin{proof}
Fix some point $p \in \im(f)$ and consider the extension $F : \B^{m}(0,1) \to T$ defined by
\[
F(st) \defl \pi_p(f(s),R\max\{0,2t-1\}) \,,
\]
for $s \in S^{m-1}$, $t \in [0,1]$ and $R \defl L\pi$. Any two points in $S^{m-1}$ can be connected by a curve $\gamma$ in $S^{m-1}$ with $\operatorname{length}(\gamma) \leq \pi$. Using Lemma~\ref{lipestimate2} we get
\begin{align*}
V_\sigma(f \circ \gamma) & = \sup \sum_{i=0}^{n-1} \sigma^{-1}(d_T(f \circ \gamma(t_{i+1}),f \circ \gamma(t_i))) \\
 & \leq \sup \sum_{i=0}^{n-1} L|\gamma(t_{i+1}) - \gamma(t_i)| \\
 & \leq \pi L \,.
\end{align*}
Because $T$ is a tree, any arc $[f(s),f(s')]$ is covered by some curve $f \circ \gamma$ with $\gamma$ as above. Hence $V_\sigma([p,q]) \leq R$ for all $q \in \im(f)$ and further $F(s) = f(s)$ for all $s \in S^{m-1}$ follows from the definition of $\pi_p$. If $x,x' \in \B^{m}(0,1)$ with $|x|,|x'| \geq \frac{1}{2}$, then
\begin{align*}
d_T(F(x),F(x')) & = d_T(\pi_p(f(|x|^{-1}x),R(2|x|-1)),\pi_p(f(|x'|^{-1}x'),R(2|x'|-1))) \\
 & \leq d_T(f(|x|^{-1}x),f(|x'|^{-1}x')) + \sigma(2R|x - x'|) \\
 & \leq \sigma(L ||x|^{-1}x - |x'|^{-1}x'|) + \sigma(2R|x - x'|) \\
 & \leq 2\sigma(2\pi L|x-x'|) \,.
\end{align*}
If $|x| \geq \frac{1}{2} \geq |x'|$, then
\begin{align*}
d_T(F(x),F(x')) & = d_T(\pi_p(f(|x|^{-1}x),R(2|x|-1)),p) \\
	& \leq \sigma(R(2|x|-1)) \\
	& \leq \sigma(R(2|x|-2|x'|)) \\
	& \leq \sigma(2\pi L|x - x'|) \,.
\end{align*}
This shows the continuity property of $F$. Any arc $[p,q]$ with endpoints in $\im(f) \subset T$ is contained entirely in $\im(f)$ because $T$ is a tree and this set is connected. Hence $\im(F) = \im(f)$ by the construction of $\pi_p$.
\end{proof}

\section{H\"older maps}

In this section we want to proof Theorem~\ref{hoelderthm}. First we establish a result that connects Property~\eqref{propertyt} with currents and winding numbers.

\begin{prop}
	\label{windvanishprop}
Let $X$ be a quasi-convex metric space with $H_1(X) = 0$ or $H_1^{\Lip}(X) = 0$ and $\varphi : X \to Y$ be a H\"older continuous map of regularity $\alpha > \frac{1}{2}$. Then $\varphi$ has Property~\eqref{propertyt} if and only if $(\varphi \circ \gamma)_\# \llbracket S^1\rrbracket = 0$ for all closed Lipschitz curves $\gamma : S^1 \to X$.

Moreover, if $Y = \R^2$, then $\varphi$ has Property~\eqref{propertyt} if and only if for all closed Lipschitz curves $\gamma : S^1 \to X$, the winding number function $q \mapsto \w_{\varphi \circ \gamma}(q)$ vanishes for almost every $q \in \R^2$.
\end{prop}

\begin{proof}
First assume that $(\varphi \circ \gamma)_\# \llbracket S^1\rrbracket = 0$ for all closed Lipschitz curves $\gamma : S^1 \to X$. Note that since $X$ is quasi-convex any curve in $X$ can be uniformly approximated by Lipschitz curves. So if we show Property~\eqref{propertyt} for Lipschitz curves, we have it for all curves. Fix two points $x,x' \in X$ with $\varphi(x) \neq \varphi(x')$ and let $\eta : [0,1] \to X$ be a Lipschitz curve connecting $x$ with $x'$. By the discussion before Lemma~\ref{windingcurrents} the current $(\varphi \circ \eta)_\#\llbracket 0,1 \rrbracket \in \cD_1(Y)$ is well defined and
\[
\partial \left((\varphi \circ \eta)_\#\llbracket 0,1 \rrbracket \right) = (\varphi \circ \eta)_\# \left( \partial \llbracket 0,1 \rrbracket \right) = \llbracket \varphi(x') \rrbracket - \llbracket \varphi(x) \rrbracket \neq 0 \,.
\]
This shows that $(\varphi \circ \eta)_\#\llbracket 0,1 \rrbracket \neq 0$. A nonzero metric current $S \in \cD_1(Y)$ as defined in \cite{L} can't be supported on finitely many points because $S(f,g) = 0$ if $g$ is locally constant on $\spt(S)$, \cite[Lemma~3.2]{L}. For another argument, a finite metric space has Nagata dimension zero, but the Nagata dimension of $\spt(S)$ has to be at least the dimension of $S$ by \cite[Proposition~2.5]{Zt}. Therefore we can find a point $y \in \spt((\varphi \circ \eta)_\#\llbracket 0,1 \rrbracket) \setminus \{\varphi(x),\varphi(x')\}$. Let $\eta' : [0,1] \to X$ be any other Lipschitz curve connecting $x$ with $x'$. We define the closed Lipschitz curve $\gamma \defl \eta \ast \eta'^{-1} : S^1 \to X$. By assumption,
\[
0 = (\varphi \circ \gamma)_\# \llbracket S^1 \rrbracket = (\varphi \circ \eta)_\# \llbracket 0,1 \rrbracket - (\varphi \circ \eta')_\# \llbracket 0,1 \rrbracket \,.
\]
In particular, $y \in \spt((\varphi \circ \eta)_\#\llbracket 0,1 \rrbracket) = \spt((\varphi \circ \eta')_\#\llbracket 0,1 \rrbracket)$. By the definition of the push-forward and the support of currents it is clear that $y \in \spt((\varphi \circ \eta)_\#\llbracket 0,1 \rrbracket) \subset \im(\varphi \circ \eta)$ and also $y \in \im(\varphi \circ \eta')$. Since $\eta'$ was arbitrary, this shows Property~\eqref{propertyt} for $\varphi$.
	
Now assume that $Y = \R^2$ and $\w_{\varphi \circ \gamma} = 0$ almost everywhere for all Lipschitz curves $\gamma : S^1 \to X$. Lemma~\ref{windingcurrents} implies that $0 = \partial \llbracket \w_{\varphi \circ \gamma} \rrbracket = (\varphi \circ \gamma)_\# \llbracket S^1 \rrbracket$ and from the first part it follows that $\varphi$ has Property~\eqref{propertyt}. On the other hand, if $\varphi : X \to \R^2$ has Property~\eqref{propertyt}, it follows from Theorem~\ref{treeintro} that there is a tree $(T,d_T)$ and maps $\psi : X \to T$, $\overbar \varphi : T \to \R^2$ with $\varphi = \overbar \varphi \circ \psi$. Let $\gamma : S^1 \to X$ be a closed Lipschitz curve. Using Corollary~\ref{extensioncor} we obtain a continuous extension $\Gamma : \B^2(0,1) \to T$ of $\psi \circ \gamma$ with $\im(\Gamma) = \im(\psi \circ \gamma)$. Therefore $\im(\overbar \varphi \circ \Gamma) \subset \im(\varphi \circ \gamma)$. As a property of the mapping degree, $\degr{\overbar \varphi \circ \Gamma}{\oB^2(0,1)}{q} \neq 0$ implies that $q$ is in the image of $\overbar \varphi \circ \Gamma$. Since $\cH^2(\im(\overbar \varphi \circ \Gamma)) \leq \cH^2(\im(\varphi \circ \gamma)) = 0$ it follows 
\[
\w_{\varphi \circ \gamma}(q) = \degr{\overbar \varphi \circ \Gamma}{\oB^2(0,1)}{q} = 0 \,,
\]
for almost every $q$ (indeed for all $q \in \R^2 \setminus \im(\varphi \circ \gamma)$).

Finally assume that $\varphi : X \to Y$ has Property~\eqref{propertyt} for a general metric space $Y$. Theorem~\ref{treeintro} gives again a factorization through a tree as in the case $Y = \R^2$ above. Consider a closed Lipschitz curve $\gamma : S^1 \to X$. Since the quotient map $\psi : X \to T$ is H\"older continuous of regularity $\alpha > \frac{1}{2}$, the current $(\psi \circ \gamma)_\# \llbracket S^1 \rrbracket \in \cD_1(T)$ is well defined. Assume by contradiction that this current is nonzero. By the definition of metric currents this means that there are Lipschitz functions $g_1,g_2 : T \to \R$ with $0 \neq (\psi \circ \gamma)_\# \llbracket S^1 \rrbracket(g_1,g_2)$. Using the Lipschitz map $g = (g_1,g_2) : T \to \R^2$ this implies with Lemma~\ref{windingcurrents},
\[
0 \neq (\psi \circ \gamma)_\# \llbracket S^1 \rrbracket(g_1,g_2) = (g \circ \psi \circ \gamma)_\# \llbracket S^1 \rrbracket(x \, dy) = \int_{\R^2} \w_{g \circ \psi \circ \gamma}(q) \, dq \,.
\]
Hence $g \circ \psi : X \to \R^2$ does not have Property~\eqref{propertyt} by the case $Y = \R^2$ considered above. But $g \circ \psi$ factors though a tree by construction and therefore has property~\eqref{propertyt}, a contradiction.
\end{proof}

The assumption $\alpha > \frac{1}{2}$ is optimal in the sense that for $\eta \in \Hol^\alpha(S^1,\R^2)$ the winding number $\w_{\eta}(q)$ is defined for almost every $q \in \R^2$ precisely because $\im(\eta)$ is a set of Lebesgue measure zero. For $\alpha \leq \frac{1}{2}$ there are closed Peano curves $\eta$ with image $[0,1]^2$ for example and as such $\w_{\eta}(q)$ is not defined for any $q \in [0,1]^2$. It is also optimal for defining continuous extensions for currents to H\"older functions. For such an extension one wishes the continuity property as in Theorem~\ref{StieltjesLemma}. But it was already noticed by Young \cite{Y}, that for $\alpha \leq \frac{1}{2}$ there are sequences of smooth functions $f_n \stackrel{\alpha}{\longrightarrow} f$ and $g_n \stackrel{\alpha}{\longrightarrow} g$ such that $\int f_n \, dg_n$ doesn't converge. Proposition~\ref{windvanishprop} has some immediate consequences in combination with Theorem~\ref{treeintro}. In particular we can recover \cite[Theorem 5]{WY}. Note that we give a formulation with $H_1^{\Lip}(X)=0$ instead of $\pi_1^{\Lip}(X)=0$ which is a slightly weaker assumption by Hurewicz' theorem.

\begin{cor}
Let $\varphi : X \to Y$.
\begin{enumerate}
	\item If $X$ is a quasi-convex metric space with $H_1(X) = 0$ or $H_1^{\Lip}(X) = 0$, $Y = \R^2$, $\varphi$ is H\"older continuous of regularity $\alpha > \frac{1}{2}$ and $\cL^2(\im(\varphi)) = 0$, then $\varphi$ factors through a tree.
	\item If $X$ is a quasi-convex metric space with $H_1^{\Lip}(X) = 0$, $Y$ is purely $2$-unrectifiable and $\varphi$ is Lipschitz continuous, then $\varphi$ factors through a geodesic tree via Lipschitz maps.
\end{enumerate}
\end{cor}

\begin{proof}
The first statement is obvious because any winding number function considered in Proposition~\ref{windvanishprop} vanishes outside the image of $\varphi$. To see the second, let $\Gamma : \B^2(0,1) \to X$ be a Lipschitz map. The current $(\varphi \circ \Gamma)_\# \llbracket \B^2(0,1) \rrbracket$ is a $2$-dimensional integral current in $Y$. Since $Y$ is purely $2$-unrectifiable, $(\varphi \circ \Gamma)_\# \llbracket \B^2(0,1) \rrbracket = 0$ and hence also,
\begin{equation}
\label{boundaryterm}
0 = \partial \left((\varphi \circ \Gamma)_\# \llbracket \B^2(0,1) \rrbracket\right) = (\varphi \circ (\Gamma|_{S^1}))_\# \llbracket S^1 \rrbracket \,.
\end{equation}
Since we assume that $H_1^{\Lip}(X) = 0$ it follows from \eqref{filling} that $(\varphi \circ \gamma)_\#\llbracket S^1 \rrbracket = 0$ for any closed Lipschitz curve $\gamma : S^1 \to X$. From the estimates of the distances in Theorem~\ref{treeintro}, the maps $\psi$ and $\overbar \varphi$ are Lipschitz and by switching to the length metric on $T$ we can also assume $(T,d_T)$ to be a length space. Note here that a length metric on a tree is indeed geodesic since the minimal length is attained by the arcs.
\end{proof}

With \eqref{twothirdsformula} we can give a proof of Theorem~\ref{hoelderthm} for the case $\alpha > \frac{2}{3}$. Assume $\varphi : X \to \R^2$ satisfies $\int \w_{\varphi \circ \gamma} = 0$ for all closed Lipschitz curves $\gamma$ in $X$. First consider some Lipschitz map $\Gamma : Q = [0,1]^2 \to X$ and let $\gamma$ be the restriction of $\Gamma$ to $\partial Q$. We want to show that $(\varphi \circ \gamma)_\# \llbracket \partial Q \rrbracket = 0$. For any square $R \subset Q$, Lemma~\ref{windingcurrents} implies
\[
\int_{\partial R} (\varphi \circ \Gamma)_x \, d(\varphi \circ \Gamma)_y = \int_{\R^2} \w_{\varphi \circ \Gamma|\partial R} = 0 \,.
\]
Hence for any $f \in \Lip(\R^2)$,
\begin{align*}
(\varphi \circ \Gamma)_\# \llbracket Q \rrbracket(f\, dx \wedge dy) & = I_Q(f \circ \varphi \circ \Gamma, (\varphi \circ \Gamma)_x, (\varphi \circ \Gamma)_y) \\
 & = \lim_{m\to \infty} \sum_{R \in \cP_m(Q)} f\circ \varphi \circ \Gamma (p_R) \int_{\partial R} (\varphi \circ \Gamma)_x \, (\varphi \circ \Gamma)_y \\
 & = 0 \,.
\end{align*}
Therefore $(\varphi \circ \Gamma)_\# \llbracket Q \rrbracket = 0$ and the same must hold for its boundary $(\varphi \circ \gamma)_\# \llbracket \partial Q \rrbracket = 0$. Since we assume that $H_1^{\Lip}(X) = 0$ we get from \eqref{filling} that ${\gamma}_\#\llbracket S^1 \rrbracket = 0$ for an arbitrary closed Lipschitz curve $\gamma : S^1 \to X$. Finally, Proposition~\ref{windvanishprop} implies the first part of Theorem~\ref{hoelderthm}.

This argument doesn't work for $\alpha \in \left]\frac{1}{2}, \frac{2}{3}\right]$ because we can't define the two-dimensional current $(\varphi \circ \Gamma)_\# \llbracket Q \rrbracket$ if $\varphi$ has this lower regularity. To circumvent this problem, we construct a functional close in spirit to $I_Q$ that makes sense also for this range of $\alpha$ and allows for evaluating a smooth test-function $f$ similar to the calculation above.

\subsection{Integration with second order terms}

In this subsection we consider H\"older maps $\varphi = (\varphi_1,\varphi_2) : Q \to \R^2$ of regularity $\alpha > \frac{1}{2}$ defined on a square $Q \subset \R^2$. We first fix some notation. As in the definition of $I_Q$ let $\cP_n(Q)$ be the partition of $Q$ into $4^n$ similar squares. For any square $R \subset Q$ fix some point $p_R \in R$ (the barycenter for example) and define
\begin{align*}
	X_{R} & \defl \varphi(p_R) \in \R^2 \,, \\
	\X^1_{R} & \defl \int_{\partial R} \varphi_1 \, d\varphi_2 \in \R \,, \\
	\X^2_{R} & \defl \frac{1}{2}\left(\int_{\partial R} (\varphi_1 - \varphi_1(p_R))^2 \, d\varphi_2, \int_{\partial R} \varphi_1 \, d(\varphi_2 - \varphi_2(p_R))^2 \right) \in \R^2 \,.
\end{align*}
Note that if we choose $p_R \in \partial R$, then all these terms depend only on the values of $\varphi$ on $\partial R$. A direct computation shows that
\begin{equation}
\label{roughidentity}
\X^2_{R} = \tilde\X_{R}^{2} - X_{R} \X^1_{R} \,,
\end{equation}
where
\[
\tilde\X_{R}^{2} \defl \frac{1}{2}\left(\int_{\partial R} \varphi_1^2 \, d\varphi_2,\int_{\partial R} \varphi_1 \, d\varphi_2^2\right) \,.
\]
These boundary terms are well defined as Riemann-Stieltjes-Young integrals by Theorem~\ref{StieltjesLemma} precisely because $\alpha > \frac{1}{2}$. Moreover we have the following a priori bounds that follow directly from results in \cite{Y} stated in Theorem~\ref{StieltjesLemma}. There is a constant $C_\alpha > 0$ such that for any square $R \subset Q$,
\begin{equation}
\label{aprioriest}
|\X^1_R| \leq C_\alpha \Hol^\alpha(\varphi)^2\diam(R)^{2\alpha} \,, \quad |\X^2_R| \leq C_\alpha \Hol^\alpha(\varphi)^3\diam(R)^{3\alpha} \,.
\end{equation}
To see the second estimate we just have to note that for $i = 1,2$,
\[
\Hol^\alpha((\varphi_i - \varphi_i(p_R))^2) \leq 2\sup_{x \in R}|\varphi_i(x) - \varphi_i(p_R)|\Hol^\alpha(\varphi_i) \leq 2 \Hol^\alpha(\varphi)^2 \diam(R)^\alpha \,.
\]

For $\alpha > \frac{2}{3}$ and $f \in \Lip(\R^2,\R)$ formula \eqref{twothirdsformula} reads
\begin{equation}
\label{mainformula1}
\int_{\R^2} f(x) \deg(\varphi,Q,x) \, dx = \lim_{n\to \infty} \sum_{R \in \cP_n(Q)} f(X_R) \X_{R}^1 \,.
\end{equation}
This can be interpreted as $\int_Q f \circ \varphi \, d\varphi_1 \wedge d\varphi_2$ even though the integrand is in general an ill-defined differential form for H\"older maps $\varphi$. We want to emphasize the close connection to the Riemann-Stieltjes integral and its extension to integration along rough paths. For a curve $X \in \Hol^\alpha([0,1],\R^n)$ with $\alpha > \frac{1}{2}$, the results of Young give meaning to $\int f \circ X \, dX$ by defining
\[
\int_0^1 f\circ X \, dX \defl \lim_{|\cP| \to 0} \sum_{[s,t] \in \cP} f(X_s)(X_t - X_s) \,.
\]
In order to solve a wide range of stochastic partial differential equations Lyons introduced rough paths \cite{Ly}. In this setting it is possible to integrate along rough paths of lower regularity $\alpha \leq \frac{1}{2}$. For instance if $\mathbf X = (X,\X)$ is a truncated rough path on $[0,1]$ of H\"older regularity $\alpha > \frac{1}{3}$, then
\begin{equation}
\label{roughpathintegral}
\int_0^1 f\circ X \, d\mathbf X \defl \lim_{|\cP| \to 0} \sum_{[s,t] \in \cP} f(X_s)(X_t - X_s) + Df(X_s)\X_{s,t}
\end{equation}
is well defined for $f \in C^{1,1}$, see e.g.\ \cite[Theorem~4.4]{FH} for a proof and defining properties of $\mathbf X$ in the notation used above. Similarly, in order to extend \eqref{mainformula1} for maps $\varphi$ with H\"older regularity $\alpha > \frac{1}{2}$ we make use of the second order terms $\X^2$. The following result is very similar to \eqref{roughpathintegral} but for maps defined on a square. 

\begin{thm}
\label{roughsurface}
Let $\varphi \in \Hol^\alpha(Q,\R^2)$. If $\alpha > \frac{2}{3}$ and $f \in \Lip(\R^2,\R)$, then
\[
\int_{\R^2} f(x) \deg(\varphi,Q,x) \, dx = \lim_{n\to \infty} \sum_{R \in \cP_n(Q)} f(X_R) \X_{R}^1 \,.
\]
If $\alpha > \frac{1}{2}$ and $f \in C^{1,1}(\R^2,\R)$, then
\[
\int_{\R^2} f(x) \deg(\varphi,Q,x) \, dx = \lim_{n\to \infty} \sum_{R \in \cP_n(Q)} f(X_R) \X_{R}^1 + Df(X_R)\X_{R}^2 \,.
\]
\end{thm}

\begin{proof}
The identity for $\alpha > \frac{2}{3}$ is just a restatement of \eqref{mainformula1}. In order to show the second identity fix some $f \in C^{1,1}(\R^2,\R)$ and let $I_n(f,\varphi)$ be the $n$th approximation on the right-hand side of the second identity in the theorem. We first show that $(I_n(f,\varphi))$ is a Cauchy sequence in $\R$ and hence its limit exists. Let $R \subset Q$ be some square and denote by $R_1,R_2,R_3,R_4$ the partition of $R$ into four squares half the size. To simplify notation we will write the terms $p_R,X_R,\X^*_R$ without index and those relating to the square $R_i$ with index $i$. Let $L \geq 0$ be a common Lipschitz constant for $f$ and $Df$ and $H$ be an upper bound on the H\"older constant of $\varphi$ with respect to $\alpha$. We have
\begin{align}
	\nonumber
&	\max_i\{|f(X) - f(X_{i})|, \|Df(X) - Df(X_{i})\|\} \\
	\label{lipest1}
	& \qquad \qquad \leq L|X - X_{i}| \leq LH|p - p_{i}|^\alpha \leq LH\diam(R)^\alpha \,.
\end{align}
The mean value theorem implies that for any $i$ there is some $x_i \in [X,X_{i}]$ such that $f(X_{i}) - f(X) = Df(x_i)(X_{i}-X)$. Hence
\begin{align}
	\nonumber
|f(X) - f(X_{i}) + D f(X)(X_{i} - X)| & = |(D f(X) - Df(x_i))(X_{i} - X)| \\
	\label{lipest2}
 & \leq L|X_{i} - X|^2 \leq LH^2\diam(R)^{2\alpha} \,.
\end{align}
	
With \eqref{roughidentity} we can rewrite 
\begin{align*}
f(X) \X^1 + Df(X)\X^2 & = f(X) \X^1 + Df(X)(\tilde \X^{2} - X \X^1) \\
 & = (f(X) - D f(X)X) \X^1 + Df(X)\tilde \X^{2} \,.
\end{align*}
Since $\sum_{i=1}^4 \X_{i}^1 = \X^1$ and $\sum_{i=1}^4 \tilde \X_{i}^2 = \tilde \X^2$ we get,
\begin{align*}
& f(X) \X^1 + Df(X)\X^2 - \sum_{i=1}^4 f(X_{i}) \X_{i}^1 + Df(X_{i})\X_{i}^2 \\
& \quad = \sum_{i=1}^4 (f(X) - D f(X)X - (f(X_{i}) - D f(X_{i})X_{{i}})) \X_{i}^1 + (Df(X) - Df(X_{i}))\tilde\X_{i}^2 \\
& \quad = \sum_{i=1}^4 (f(X) - f(X_{i}) + Df(X)(X_{i} - X)) \X_i^1 \\
& \quad \qquad \quad + (D f(X_{i}) - Df(X))X_{{i}} \X_{i}^1 + (Df(X) - Df(X_{i}))\tilde\X_i^2 \\
& \quad = \sum_{i=1}^4 (f(X) - f(X_{i}) + Df(X)(X_{i} - X)) \X_{i}^1  + (Df(X) - Df(X_{i}))\X_{i}^2 \,.
\end{align*}
Applying the estimates in \eqref{lipest1}, \eqref{lipest2} and \eqref{aprioriest} to the identity above leads to
\begin{align*}
\Biggl|f(X) \X^1 & + Df(X)\X^2 - \sum_{i=1}^4 f(X_{i}) \X_{i}^1 + Df(X_{i})\X_{i}^2\Biggr| \\
 & \leq \sum_{i=1}^4 LH^2\diam(R)^{2\alpha}|\X_{i}^1| + LH\diam(R)^\alpha|\X_{i}^2| \\
 & \leq 8C_\alpha LH^4\diam(R)^{4\alpha} \,.
\end{align*}
The difference of successive approximations is estimated by summing over all $R \in \cP_n(Q)$ for some fixed $n \in \N$. This gives
\begin{equation}
\label{squareestimate}
|I_n(f,\varphi) - I_{n+1}(f,\varphi)| \leq 4^n CH^4 2^{-4\alpha n} = CH^42^{2n(1 - 2\alpha)} \,,
\end{equation}
for some constant $C$ depending on $\diam(Q)$, $\alpha$ and $L$. Since $\alpha > \frac{1}{2}$ we see that $(I_n(f,\varphi))$ is a Cauchy sequence in $\R$ and hence converges. Denote by $I(f,\varphi)$ its limit. From \eqref{squareestimate} it is straightforward to obtain the upper bound
\begin{equation}
\label{approximation}
|I(f,\varphi) - I_n(f,\varphi)| \leq C'H^4 2^{2n(1 - 2\alpha)} \,,
\end{equation}
for a constant $C'$ with the same dependencies as $C$. Let $(\varphi_k)$ be a sequence in $\Hol^\alpha(Q,\R^2)$ with $\varphi_k \to \varphi$ and $\sup_k \Hol^\alpha(\varphi_k) \leq H$. As finite sums over Riemann-Stieltjes-Young integrals and evaluations of $\varphi_k$ it follows from Theorem~\ref{StieltjesLemma} that $\lim_{k\to\infty}I_n(f,\varphi_k) = I_n(f,\varphi)$ for all $n$. Applying \eqref{approximation} we obtain
\begin{align*}
\limsup_{k\to\infty} |I(f,\varphi_k) - I(f,\varphi)| & \leq \limsup_{k\to\infty} \bigl(|I(f,\varphi_k) - I_n(f,\varphi_k)| \\
 & \qquad + |I_n(f,\varphi_k) - I_n(f,\varphi)| + |I_n(f,\varphi) - I(f,\varphi)| \bigr) \\
 & \leq 2C'H^4 2^{2n(1 - 2\alpha)} \,.
\end{align*}
This is true for all $n$ and hence $I(f,\varphi_k)$ converges to $I(f,\varphi)$. 
	
For Lipschitz maps the sum involving the terms $\X_{R}^2$ vanish in the limit since $|\X_{R}^2| \leq o(\diam(R)^2)$ by \eqref{aprioriest} and hence for Lipschitz maps the second formula in the theorem is a consequence of the first one. It follows as in the proof of \cite[Proposition~4.6]{Zt} that $x \mapsto \deg(\varphi,Q,x)$ is integrable in case $\varphi$ is in $\Hol^\alpha$ and any $\varphi$ can be approximated by a sequence of Lipschitz maps $(\varphi_k)$ in such a way that $\sup_k\Hol^\alpha(\varphi_k) < \infty$ and $x \mapsto \deg(\varphi_k,Q,x)$ converges weakly (as distributions in the sense of Schwartz) to $x \mapsto \deg(\varphi,Q,x)$. This proves the theorem.
\end{proof}

The setting in the theorem above is simpler than for rough paths in \eqref{roughpathintegral} because range and domain have the same dimension and we get the terms $\X^2$ for free. This fact shouldn't come as a surprise since a compactly supported $2$-dimensional current in $\R^2$ is uniquely determined by its boundary due to the constancy theorem for currents \cite{F} and we already know from Lemma~\ref{windingcurrents} that the filling of $\varphi_\#\llbracket \partial Q\rrbracket$ is obtained by integrating over the degree of $\varphi$. The situation for maps into higher dimensional spaces is different and new ideas are needed to extend Theorem~\ref{roughsurface} in case $\alpha \in \left]\frac{1}{2}, \frac{2}{3}\right]$ and the target $\R^2$ is replaced by $\R^n$ for $n \geq 3$.

With this construction we can give a proof of Theorem~\ref{hoelderthm} stated in the introduction.

\begin{proof}[Proof of Theorem~\ref{hoelderthm}]
If $\varphi$ has Property~\eqref{propertyt}, then Proposition~\ref{windvanishprop} implies that all the integrals in the statement of the theorem vanish.

To see the converse implication, assume that all the winding number integrals vanish as stated. Fix $Q = [0,1]^2$, let $\Gamma : Q \to X$ be some Lipschitz map and set $\tilde\varphi \defl \varphi \circ \Gamma : Q \to \R^2$ and $\gamma \defl \Gamma|_{\partial Q}$. By assumption and Lemma~\ref{windingcurrents} we have for all squares $R \subset Q$,
\[
\int_{\partial R} \tilde\varphi_x \, d\tilde\varphi_y = 0 \,,
\]
in case $\alpha > \frac{2}{3}$ and
\[
\int_{\partial R} \tilde\varphi_x \, d\tilde\varphi_y = \int_{\partial R} \tilde\varphi_x^2 \, d\tilde\varphi_y = \int_{\partial R} \tilde\varphi_x \, d\tilde\varphi_y^2 = 0 \,,
\]
in case $\alpha > \frac{1}{2}$. Theorem~\ref{roughsurface} shows that for smooth $f : \R^2 \to \R$,
\[
\int_{\R^2} f(q) \w_{\varphi \circ \gamma}(q) \, dq = 0 \,.
\]
Hence $\w_{\varphi \circ \gamma}$ vanishes almost everywhere and thus $\gamma_\#\llbracket \partial Q \rrbracket = 0$ by Lemma~\ref{windingcurrents}. Since we assume that $H_1^{\Lip}(X) = 0$ it follows from \eqref{filling} that $\gamma_\#\llbracket S^1 \rrbracket = 0$ holds for any closed Lipschitz curve $\gamma : S^1 \to X$. With Proposition~\ref{windvanishprop} we conclude that $\varphi$ has Property~\eqref{propertyt}.
\end{proof}

\subsection{Heisenberg group target}

The first Heisenberg group equipped with the Carnot-Carath\'eodory metric $(\He,d_{\cc})$ is bi-Lipschitz equivalent to $\R^3$ equipped with the \emph{Kor\'anyi metric}, 
\[
d_K(p,q) \defl \left[(|q_x-p_x|^2 + |q_y-p_y|^2)^2 + 16|q_z - p_z - \tfrac{1}{2}(p_xq_y - p_yq_x)|^2 \right]^\frac{1}{4} \,,
\]
see e.g.\ \cite[Subsection~2.2.1]{CDPT}. Since the statements of Theorem~\ref{heisenberggroupthm} do not depend on a change to a bi-Lipschitz equivalent metric, we will work with $(\R^3,d_{\text{K}})$ instead of $(\He,d_{\cc})$. It is rather direct to check that for any bounded subset $B \subset \R^3$ there is a constant $C > 0$ such that for all $p,q \in B$,
\[
C^{-1} d_{\Eucl}(p,q) \leq d_{\text{K}}(p,q) \leq C d_\Eucl(p,q)^\frac{1}{2} \,,
\]
where $d_\Eucl$ denotes the Euclidean metric. Along the proof of \cite[Lemma 3.2]{LZ} one can show that for curves $\gamma : [a,b] \to (\R^3,d_{\text{K}})$ and $f : [a,b] \to \R$ of H\"older regularity $\alpha > 1/2$,
\begin{equation}
	\label{lifting}
	\int_a^b f \, d\gamma_z = \frac{1}{2}\left[\int_a^b f \gamma_x \, d\gamma_y - \int_a^b f \gamma_y \, d\gamma_x \right] \,.
\end{equation}
If $\gamma$ is further a closed curve, this implies
\begin{equation}
\label{lifting2}
\int \gamma_x \, d\gamma_z = \frac{3}{4}\int \gamma_x^2 \, d\gamma_y, \qquad \int \gamma_y \, d\gamma_z = \frac{3}{4}\int \gamma_x \, d\gamma_y^2 \,.
\end{equation}
Here is a derivation of the first identity,
\begin{align*}
\int \gamma_x \, d\gamma_z & = \frac{1}{2}\left[\int \gamma_x^2 \, d\gamma_y - \int \gamma_x\gamma_y \, d\gamma_x \right] \\
 & = \frac{1}{2}\left[\int \gamma_x^2 \, d\gamma_y - \frac{1}{2}\int \gamma_y \, d\gamma_x^2 \right] = \frac{3}{4}\int \gamma_x^2 \, d\gamma_y \,.
\end{align*}
It is interesting to note that the terms $\int \gamma_x^2 \, d\gamma_y$ and $\int \gamma_x \, d\gamma_y^2$ that appear in \eqref{lifting2} are precisely those that are assumed to vanish in Theorem~\ref{hoelderthm} in case $\frac{1}{2} < \alpha \leq \frac{2}{3}$.
First we show the following lemma. 

\begin{lem}
	\label{heisenbergsquare}
Let $Q \subset \R^2$ be a square and $\varphi : (Q,d_\Eucl) \to (\R^3,d_{\text{K}})$ be H\"older continuous of regularity $\alpha > \frac{2}{3}$. Then $\tilde \varphi_\# \llbracket Q \rrbracket = 0$ for the H\"older map $\tilde \varphi : (Q,d_\Eucl) \to (\R^3,d_\Eucl)$ obtained by changing the metric on $\R^3$.
\end{lem}

\begin{proof}
By a smoothing argument it is enough to show that $\tilde \varphi_\# \llbracket Q \rrbracket(\omega) = 0$ for any smooth differential form $\omega \in \Omega^2(\R^3)$. Using \eqref{twothirdsformula} and since $\omega$ can be written as $\sum_{i < j} g_{ij} \, dx_i \wedge dx_j$ for smooth functions $g_{ij}$ on $\R^3$, it is enough to show that for all indices $i < j$,
\[
\tilde \varphi_\# \llbracket Q \rrbracket(g_{ij} \, dx_i \wedge dx_j) = I_Q(g_{ij} \circ \varphi, \varphi_i,\varphi_j) = 0 \,.
\]
By the definition of $I_Q$ and \eqref{lifting2},
\begin{align*}
\tilde \varphi_\# \llbracket Q \rrbracket(g \, dx \wedge dz) & = I_Q(g \circ \varphi, \varphi_x,\varphi_z) \\
 & = \lim_{n \to \infty} \sum_{R \in \cP_n(Q)} g\circ\varphi(p_R) \int_{\partial R} \varphi_x \, d\varphi_z \\
 & = \lim_{n \to \infty} \frac{3}{4} \sum_{R \in \cP_n(Q)} g\circ\varphi(p_R) \int_{\partial R} \varphi_x^2 \, d\varphi_y \\
 & = \frac{3}{4} \tilde \varphi_\# \llbracket Q \rrbracket(g \, dx^2 \wedge dy) \\
 & = \frac{3}{2} \tilde \varphi_\# \llbracket Q \rrbracket(xg \, dx \wedge dy) \,.
\end{align*}
Similarly, $\tilde \varphi_\# \llbracket Q \rrbracket(g \, dy \wedge dz) = \frac{3}{2} \tilde \varphi_\# \llbracket Q \rrbracket(yg \, dx \wedge dy)$ and hence it remains to show that $\tilde \varphi_\# \llbracket Q \rrbracket(g \, dx \wedge dy) = 0$ for all smooth $g : \R^3 \to \R$. By setting $f \equiv 1$ in \eqref{lifting}, we get $\int_{\partial R} \varphi_x \, d\varphi_y = 0$ for all squares $R \subset Q$, and therefore $\tilde \varphi_\# \llbracket Q \rrbracket(g \, dx \wedge dy) = 0$ follows from the definition of $I_Q$.
\end{proof}

With this preparation we can give a proof of the remaining theorem in the introduction.

\begin{proof}[Proof of Theorem~\ref{heisenberggroupthm}]
Let $\varphi : (X,d_X) \to (\He,d_{\cc})$ be a H\"older map as in the statement of the theorem. In order to apply Theorem~\ref{treeintro} we will show that $\tilde \varphi : X \to \R^3$ as defined in the lemma above has Property~\eqref{propertyt}. Let $\gamma : \partial Q \to X$ be any closed Lipschitz curve defined on the boundary of some square $Q \subset \R^2$ and assume that there is a Lipschitz extension $\Gamma : Q \to X$. By Lemma~\ref{heisenbergsquare},
\[
0 = \partial((\tilde \varphi \circ \Gamma)_\# \llbracket Q \rrbracket) = (\tilde \varphi \circ \gamma)_\# \llbracket \partial Q \rrbracket \,.
\]
Since we assume $H_1^{\Lip}(X) = 0$ the same result for arbitrary closed Lipschitz curves is a consequence of \eqref{filling}. Proposition~\ref{windvanishprop} now implies that $\tilde \varphi$ has Property~\eqref{propertyt}. Because this property is purely topological, the same holds for $\varphi$ and Theorem~\ref{treeintro} applies.
\end{proof}

With Corollary~\ref{extensioncor} we can conclude that the higher homotopy groups $\pi^{\alpha}_k(\He,d_{\cc})$ for $k \geq 2$ are trivial in the category of H\"older continuous maps with regularity $\alpha > \frac{2}{3}$ similar to the conclusion in \cite{WY} with respect to Lipschitz maps.


\end{document}